\documentclass[reqno]{amsart}
\usepackage{mathrsfs}
\usepackage{color}
\usepackage{amsmath}
\usepackage{amsfonts}
\usepackage{amssymb}
\usepackage{graphicx}
\usepackage{hyperref}

 \newtheorem{Theorem}{Theorem}[section]
 \newtheorem{Corollary}[Theorem]{Corollary}
 \newtheorem{Lemma}[Theorem]{Lemma}
 \newtheorem{Proposition}[Theorem]{Proposition}

 \newtheorem{Definition}[Theorem]{Definition}

 \newtheorem{Remark}[Theorem]{Remark}

 \newtheorem{Example}[Theorem]{Example}
 \numberwithin{equation}{section}

\begin{document}

\title[On the multipoled global Zhou weights and Zhou numbers]
 {On the multipoled global Zhou weights and semi-continuity for Zhou numbers}

\author{Shijie Bao}
\address{Shijie Bao: Institute of Mathematics, Academy of Mathematics and Systems Science, Chinese Academy of Sciences, Beijing 100190, China.}
\email{bsjie@amss.ac.cn}

\author{Qi'an Guan}
\address{Qi'an Guan: School of
Mathematical Sciences, Peking University, Beijing 100871, China.}
\email{guanqian@math.pku.edu.cn}

\author{Zhitong Mi}
\address{Zhitong Mi: School of Mathematics and Statistics, Beijing Jiaotong University, Beijing, 100044, China.}
\email{zhitongmi@amss.ac.cn}

\author{Zheng Yuan}
\address{Zheng Yuan: Institute of Mathematics, Academy of Mathematics and Systems Science, Chinese Academy of Sciences, Beijing 100190, China.}
\email{yuanzheng@amss.ac.cn}

\thanks{}

\subjclass[2020]{Primary: 32U35 Secondary: 14B05 32U15 32U25}

\keywords{multipoled global Zhou weight, Zhou number, strong openness property, semi-continuity}

\date{}

\dedicatory{}

\commby{}


\begin{abstract}
In the present paper, we give the definition and properties of the multipoled global Zhou weights. Some approximation and convergence results of multipoled global Zhou weights are given. We also establish a semi-continuity result for the Zhou numbers.
\end{abstract}

\maketitle
\section{Introduction}

\subsection{Background}

Measuring the singularity of a plurisubharmonic function near its pole is an important problem in several complex variables, complex geometry and algebraic geometry. Many mathematical concepts were established to do this, such as the Lelong numbers (\cite{Lel57}), Kiselman numbers (\cite{Ki87}), Lelong-Demailly numbers (\cite{Dem82}), and relative types (\cite{Rash06}). The Lelong numbers and Kiselman numbers have some good properties, called tropical multiplicativity ann tropical additivity (see \cite{CADG,Rash06}). For any plurisubharmonic function $u$ near the origin $o$ in $\mathbb{C}^n$, the Lelong number of $u$ at $o$ is defined by:
\[\nu(u,o):=\sup\big\{u\le c\log|z|+O(1) \ \text{near} \ o\big\}.\]
It can be easily checked that for any plurisubharmonic function $u$ and $v$ near $o$, the following properties hold for the Lelong number:
\begin{flalign*}
    \begin{split}
        \nu(u+v,o)=\nu(u,o)+\nu(v,o) \  &(\emph{\text{tropical multiplicativity}})\\
        \nu(\max\{u,v\},o)=\min\{\nu(u,o),\nu(v,o)\} \ &(\emph{\text{tropical additivity}}).
    \end{split}
\end{flalign*}

However, the Lelong numbers and Kiselman numbers are still very rough for measuring the singularities of plurisubharmonic functions. Meanwhile, in general, the Lelong-Demailly numbers do not satisfy tropical additivity, and the relative types do not satisfy tropical multiplicativity.

The relative types were introduced by Rashkovskii, which measure the plurisubharmonic functions by the maximal weights. More precisely, let $\varphi$ be a \emph{maximal weight} near $o$, i.e., there exists an open neighborhood $U$ of $o$ such that $(dd^c\varphi)^n\equiv 0$ on $U\setminus\{o\}$, then for any plurisubharmonic function $u$ near $o$, the \emph{relative type} of $u$ to $\varphi$ is defined as:
\[\sigma_{o}(u,\varphi):=\sup\big\{c\ge 0 : u\le c\varphi+O(1) \ \text{near} \ o \big\}.\]
The maximality of $\varphi$ verifies that $u\le\sigma_o(u,\varphi)\varphi+O(1)$ near $o$ (see \cite{Rash06}).

Since the relative types to maximal weights are not tropically multiplicative in general, it may be better to narrow the range of maximal weights considered. For this purpose, \cite{BGMY23} introduced a class of tame maximal weights, called \emph{Zhou weights}, including the local version and global version, whose relative types satisfy both the tropical multiplicativity ann tropical additivity (see Corollary \ref{cor-tropical.multi.additi}). The relative types to the Zhou weights are also called \emph{Zhou numbers} specifically. We recall the definitions of the local Zhou weights and global weights here, and a more detailed review of them are given in Section \ref{sec-preliminaries}.

Let $f_{0}=(f_{0,1},\cdots,f_{0,m})$ be a vector,
where $f_{0,1},\cdots,f_{0,m}$ are holomorphic functions near the origin $o\in\mathbb{C}^n$. Denote $|f_{0}|^{2}=|f_{0,1}|^{2}+\cdots+|f_{0,m}|^{2}$. Let $\varphi_{0}$ be a plurisubharmonic function near $o$,
such that $|f_{0}|^{2}e^{-2\varphi_{0}}$ is integrable near $o$.

\begin{Definition}[Definition 1.2 in \cite{BGMY23}]\label{def-max.relat.in.intro}
    We call that $\phi^{f_0,\varphi_0}_{o,\max}$ ($\phi_{o,\max}$ for short) is a \textbf{local Zhou weight related to $|f_{0}|^{2}e^{-2\varphi_{0}}$ near $o$},
    if the following three statements hold:
    
    (1) $|f_{0}|^{2}e^{-2\varphi_{0}}|z|^{2N_{0}}e^{-2\phi_{o,\max}}$ is integrable near $o$ for large enough $N_{0}\gg 0$;
    
    (2) $|f_{0}|^{2}e^{-2\varphi_{0}}e^{-2\phi_{o,\max}}$
    is not integrable near $o$;
    
    (3) for any plurisubharmonic function $\varphi'\geq\phi_{o,\max}+O(1)$ near $o$ such that $|f_{0}|^{2}e^{-2\varphi_{0}}e^{-2\varphi'}$
    is not integrable near $o$, $\varphi'=\phi_{o,\max}+O(1)$ holds.
\end{Definition}

The global Zhou weights were defined to verify that any local Zhou weight can be replaced by a continuous and maximal plurisubharmonic function with the equivalent singularity. For the definition, first, let $D$ be a domain in $\mathbb{C}^n$ with $o\in D$. Addtionally, let $f_0$ and $\varphi_0$ be a holomorphic vector and a plurisubharmonic function near $o$ respectively as before the definition of the local Zhou weight. 
	\begin{Definition}[Definition 1.15 in \cite{BGMY23}]\label{def-global.Zhou.weight.in.intro}
		We call that a negative plurisubharmonic function $\Phi^{f_0,\varphi_0,D}_{o,\max}$ ($\Phi^{D}_{o,\max}$ for short) on $D$ a \textbf{global Zhou weight related to $|f_0|^2e^{-2\varphi_0}$ on $D$} if the following statements hold:
		
		$(1)$ $|f_0|^2e^{-2\varphi_0}|z|^{2N_0}e^{-2\Phi^{D}_{o,\max}}$ is integrable near $o$ for large enough $N_0\gg 0$;
		
		$(2)$ $|f_0|^2e^{-2\varphi_0-2\Phi^{D}_{o,\max}}$ is not integrable near $o$;
		
		$(3)$ for any negative plurisubharmonic function $\tilde\varphi$ on $D$ satisfying that $\tilde\varphi\ge\Phi^{D}_{o,\max}$ on $D$ and $|f_0|^2e^{-2\varphi_0-2\tilde\varphi}$ is not integrable near $o$, $\tilde\varphi=\Phi^{D}_{o,\max}$ holds on $D$.
	\end{Definition}

    The construction of the Zhou weights in \cite{BGMY23} relies on the strong openness property of multiplier ideal sheaves. For any plurisubharmonic function $u$ near $z\in\mathbb{C}^n$, the stalk of the multiplier ideal sheaf $\mathcal{I}(u)$ at $z$ is defined as (see \cite{Na90, AMAG}):
    \[\mathcal{I}(u)_{z}:=\big\{(f,z)\in\mathcal{O}_z : |f|^2e^{-2u} \ \text{is integrable near} \ z\big\}.\]
    The strong openness property (\cite{GZ15a}) says that
    \[\mathcal{I}(u)=\mathcal{I}_+(u):=\bigcup_{p>1}\mathcal{I}(pu).\]
    More detailed statements of the strong openness property can be seen in Section \ref{subsec-required.lemma}.

It is verified in \cite{BGMY23} that the global Zhou weights are generalizations of the pluricomplex Green functions on hyperconvex domain, while the multipoled pluricomplex Green functions are on the another direction to generalize the pluricomplex Green function, which have several distinct poles on hyperconvex domains.

Let $D$ be a hyperconvex domain, and $\mathbf{Z}:=\{\mathbf{z}_1,\ldots,\mathbf{z}_p\}$ be $p$ distinct points in $D$, $p\in\mathbb{N}_+$. The \emph{multipoled pluricomplex Green function} $G_{D,(\mathbf{z}_1,\ldots,\mathbf{z}_p)}$ ($G_{D,\mathbf{Z}}$ for simplicity) on $D$ with the poles $\mathbf{z}_1,\ldots,\mathbf{z}_p$ is defined by:
\begin{Definition}[multipoled pluricomplex Green function, \cite{Lel89}]\label{def-multipoled.Green.function}
    For any $z\in D$,
    \begin{align*}
        G_{D,(\mathbf{z}_1,\ldots,\mathbf{z}_p)}(z):=\sup\big\{\phi(z) : \phi\in \mathrm{PSH}^-(D), \ \text{and} \ \phi(\cdot)\le \log|\cdot-\mathbf{z}_i|+O(1)& \\
        \text{near any } \mathbf{z}_i, \ i=1,\ldots, p&\big\}.
    \end{align*}
\end{Definition}

It is well-known that the multipoled pluricomplex Green function satisfies:

\emph{$e^{G_{D,\mathbf{Z}}}$ is continuous on $D$, $G_{D,\mathbf{Z}}$ is exhaustive on $D$, $G_{D,\mathbf{Z}}$ is maximal on $D\setminus\mathbf{Z}$, and
\[G_{D,\mathbf{Z}}(z)=\log|z-\mathbf{z}_i|+O(1) \ \text{near} \ \mathbf{z}_i, \ \forall i=1,\ldots,p.\]}

The multipoled pluricomplex Green functions inspire us to consider the multipoled global Zhou weights on hyperconvex domains.

We give the definition and basic properties of multipoled global Zhou weights on hyperconvex domains in the present paper, including generalizing the properties of multipoled pluricomplex Green functions mentioned above. We also consider the convergence of the multipoled global Zhou weights when the poles move, where the corresponding result of pluricomplex Green functions were proved by Demailly \cite{Dem87b}.   

The celebrated Siu's semi-continuity theorem (\cite{Siu74}) of Lelong numbers is important and widely used. As generalizations of Lelong numbers, the corresponding semi-continuity result of Zhou numbers is worth of being established. We will prove a semi-continuity result of Zhou numbers in the present paper.

\subsection{Multipoled global Zhou weights}\label{subsec.basic.property}
Let $D$ be a hyperconvex domain in $\mathbb{C}^n$. Denote by $\mathrm{PSH}^-(D)$ the set of all negative plurisubharmonic functions on $D$. Suppose we are given a priori of data
\[\mathscr{P}=\bigg(\{\mathbf{z}_i\}_{i=1}^p, \{\mathbf{f}_{0,i}\}_{i=1}^p, \{u_{0,i}\}_{i=1}^p\bigg)\]
for some $p\in\mathbb{N}_+$, which satisfies:

(a) $\mathbf{z}_i$ ($1\le i\le p$) are distinct points in $D$;

(b) $\mathbf{f}_{0,i}=(f_{0,i,1},\ldots,f_{0,i,k_i})$ is a holomorphic vector near $\mathbf{z}_i$ for any $i$, where $f_{0,i,j}\in \mathcal{O}_{\mathbf{z}_i}$, $\forall 1\le i\le p$, $1\le j\le k_{i}$;

(c) $u_{0,i}$ is a plurisubharmonic function near $\mathbf{z}_i$ such that $|\mathbf{f}_{0,i}|^2e^{-2u_{0,i}}$ is $L^1$ integrable near $\mathbf{z}_i$ for any $i$, where $|\mathbf{f}_{0,i}|^2:=|f_{0,i,1}|^2+\cdots+|f_{0,i,k_i}|^2$.

\begin{Definition}\label{def-Zhou.maximal.weight}
We call $\Phi^D_{\mathscr{P},\max}\in\mathrm{PSH}^-(D)$ a \textbf{global Zhou weight related to $\mathscr{P}$ on $D$}, if the following statements hold:

(1) for sufficiently large $N_1, \ldots, N_p>0$,
\[|\mathbf{f}_{0,i}|^2e^{-2u_{0,i}}|z-\mathbf{z}_i|^{2N_i}e^{-2\Phi_{\mathscr{P},\max}^D}\]
is integrable near $\mathbf{z}_i$ for any $i=1,\ldots,p$;

(2) $|\mathbf{f}_{0,i}|^2e^{-2u_{0,i}-2\Phi^D_{\mathscr{P},\max}}$ is not integrable near $\mathbf{z}_i$ for any $i=1,\ldots,p$;

(3) for any $\Psi\in\mathrm{PSH}^-(D)$ satisfying that $\Psi\ge\Phi_{\mathscr{P},\max}^D$ on $D$ and $|\mathbf{f}_{0,i}|^2e^{-2u_{0,i}-\Psi}$ is not integrable near $\mathbf{z}_i$ for any $i$, $\Psi=\Phi_{\mathscr{P},\max}^D$ holds on $D$. 
\end{Definition}

The following remark verifies the existence of multipoled global Zhou weights. The existence can also be seen by Proposition \ref{prop-local.produce.global}. 
\begin{Remark}\label{rem-existence}
    Assume that there exists $\Phi\in\mathrm{PSH}^-(D)$ such that for sufficiently large $N_1,\ldots, N_p>0$,
    \[|\mathbf{f}_{0,i}|^2e^{-2u_{0,i}}|z-\mathbf{z}_i|^{2N_i}e^{-2\Phi}\]
    is integrable near $\mathbf{z}_i$, and $(\mathbf{f}_{0,i},\mathbf{z}_i)\notin\mathcal{I}(u_{0,i}+\Phi)_{\mathbf{z}_i}$ for any $i=1,\ldots,p$. Then there exists a global Zhou weight $\Phi_{\mathscr{P},\max}^D$ on $D$ related to $\mathscr{P}$ such that $\Phi_{\mathscr{P},\max}^D\ge \Phi$ on $D$.

    Moreover, there exists $N>0$ such that
    \begin{equation}\label{eq-Phi.ge.NG_D}
    \Phi_{\mathscr{P},\max}^D(z)\ge NG_{D,(\mathbf{z}_1,\ldots,\mathbf{z}_p)}(z), \ \forall z\in D,
    \end{equation}
    where $G_{D,(\mathbf{z}_1,\ldots,\mathbf{z}_p)}(z)$ is the multipoled pluricomplex Green function on $D$ with the poles $\mathbf{z}_1,\ldots,\mathbf{z}_p$ (see Definition \ref{def-multipoled.Green.function}).
\end{Remark}

Near each pole, the multipoled is also a local Zhou weight. The following proposition shows this fact.

\begin{Proposition}\label{prop-global.is.local}
Let $\Phi_{\mathscr{P},\max}^D$ be a global Zhou weight related to the priori $\mathscr{P}$ on $D$, then $\Phi_{\mathscr{P},\max}^D$ is a local Zhou weight related to $|\mathbf{f}_{0,i}|^2e^{-2u_{0,i}}$ near $\mathbf{z}_i$ for any $i=1,\ldots,p$.
\end{Proposition}

The following proposition states that we can produce a global Zhou weight with several specific poles by some local Zhou weights at each pole.

\begin{Proposition}\label{prop-local.produce.global}
    Let $\phi_{\mathbf{z}_i,\max}$ be a local Zhou weight related to some $|\mathbf{f}_{0,i}|^2e^{-2u_{0,i}}$ near $\mathbf{z}_i$ for any $i=1,\ldots,p$. Then the function
    \begin{flalign*}
        \begin{split}
            \Phi_{\mathscr{P},\max}^D(z):=\sup\big\{\phi(z) :\ &\phi\in\mathrm{PSH}^-(D), \ (\mathbf{f}_{0,i},\mathbf{z}_i)\notin\mathcal{I}(u_{0,i}+\phi)_{\mathbf{z}_i},\\
            &\text{and} \ \phi\ge \phi_{\mathbf{z}_i,\max}+O(1) \ \text{near} \ \mathbf{z}_i, \ \forall i=1,\ldots,p \big\},
        \end{split}
        \end{flalign*}
is a global Zhou weight on $D$ related to
\[\mathscr{P}=\bigg(\{\mathbf{z}_i\}_{i=1}^p, \{\mathbf{f}_{0,i}\}_{i=1}^p, \{u_{0,i}\}_{i=1}^p\bigg).\]
Moreover, for any $i\in\{1,\ldots,p\}$, the function
\[\Phi_{\mathscr{P},\max}^D=\phi_{\mathbf{z}_i,\max}+O(1) \ \text{near} \ \mathbf{z}_i.\]
\end{Proposition}

Combining Proposition \ref{prop-global.is.local} and Proposition \ref{prop-local.produce.global}, we can get that the following corollary holds.
\begin{Corollary}
    Let $\Phi_{\mathscr{P},\max}^D$ be a global Zhou weight related to the priori $\mathscr{P}$ on $D$, then for any $z\in D$,
    \begin{flalign*}
        \begin{split}
            \Phi_{\mathscr{P},\max}^D(z)=\sup\big\{\phi(z) : \ &\phi\in\mathrm{PSH}^-(D), \ (\mathbf{f}_{0,i},\mathbf{z}_i)\notin\mathcal{I}(u_{0,i}+\phi)_{\mathbf{z}_i},\\
            &\text{and} \ \phi\ge\Phi_{\mathscr{P},\max}^D+O(1) \ \text{near} \ \mathbf{z}_i, \ \forall i=1,\ldots,p \big\},
        \end{split}
        \end{flalign*}
\end{Corollary}

For any $\psi\in\mathrm{PSH}(D)$, we denote
     \[\sigma_{\mathbf{Z}}\big(\psi,\Phi_{\mathscr{P},\max}^D\big):=\min_{1\le i\le p}\sigma_{\mathbf{z}_i}\big(\psi,\Phi_{\mathscr{P},\max}^D\big),\]
where
\[\sigma_{\mathbf{z}_i}\big(\psi,\Phi_{\mathscr{P},\max}^D\big)=\sup\big\{c\ge 0 : \psi\le c\Phi_{\mathscr{P},\max}^D+O(1) \ \text{near} \ \mathbf{z}_i\big\}.\]

\begin{Proposition}\label{prop-psi.le.sigmaZ.Phi}
    Let $\Phi_{\mathscr{P},\max}^D$ be a global Zhou weight related to $\mathscr{P}$ on $D$. Then for any $\psi\in\mathrm{PSH}^-(D)$, the following inequality holds on $D$:
    \[\psi\le\sigma_{\mathbf{Z}}\big(\psi,\Phi_{\mathscr{P},\max}^D\big)\Phi_{\mathscr{P},\max}^D.\]
    
\end{Proposition}

The local boundedness and maximality of multipoled global Zhou weights are established in Proposition \ref{prop-maximal}, while Proposition \ref{prop-continuous.exhaustive} establishes the exhaustion and continuity. 

\begin{Proposition}\label{prop-maximal}
    Let $\Phi_{\mathscr{P},\max}^D$ be a global Zhou weight related to $\mathscr{P}$ on $D$. Then
    \[\Phi_{\mathscr{P},\max}^D\in L_{\mathrm{loc}}^{\infty}(D\setminus\{\mathbf{z_1},\ldots,\mathbf{z}_n\}),\]
    and
    \[\big(dd^c\Phi_{\mathscr{P},\max}^D\big)^n=0 \ \text{on} \ D\setminus\{\mathbf{z_1},\ldots,\mathbf{z}_n\}).\]
\end{Proposition}

\begin{Proposition}\label{prop-continuous.exhaustive}
    Let $\Phi_{\mathscr{P},\max}^D$ be a global Zhou weight related to $\mathscr{P}$ on $D$, where $D$ is a bounded hyperconvex domain. Then $e^{\Phi_{\mathscr{P},\max}^D}$ is continuous on $D$, and $\Phi_{\mathscr{P},\max}^D(z)\to 0$ as $z\to \partial D$.
\end{Proposition}

\begin{Example}
    By \cite[Example 1.5 and Proposition 1.6]{BGMY23}, the function $\phi=\log|z|$ is a local Zhou weight related to $e^{-(2n-1)\log|z|}$ near $o$. Then, for a hyperconvex domain $D\subset\mathbb{C}^n$, and $\mathbf{Z}:=\big\{\mathbf{z}_1,\ldots,\mathbf{z}_p\big\}\subset D$, since the multipoled pluricomplex Green function with the pole set $\mathbf{Z}$ satisfies:
    \begin{align*}
        G_{D,(\mathbf{z}_1,\ldots,\mathbf{z}_p)}(z)=\sup\big\{\phi(z) : \phi\in \mathrm{PSH}^-(D), \ \text{and} \ \phi(\cdot)=\log|\cdot-\mathbf{z}_i|+O(1)& \\
        \text{near any } \mathbf{z}_i, \ i=1,\ldots, p&\big\},
    \end{align*}
    according to Proposition \ref{prop-local.produce.global}, the function $G_{D,(\mathbf{z}_1,\ldots,\mathbf{z}_p)}(z)$ is a multipoled global Zhou weight on $D$ related to
    \[\mathscr{P}=\bigg(\{\mathbf{z}_i\}_{i=1}^p, \{1\}_{i=1}^p, \{e^{-(2n-1)\log|z-\mathbf{z}_i|}\}_{i=1}^p\bigg).\] 
\end{Example}

\subsection{Approximations of multipoled global Zhou weights}
In this section, we give results about the approximations of the global Zhou weights with several poles. 
   
Let $D$ be a hyperconvex domain in $\mathbb{C}^n$, and $\Phi_{\mathscr{P},\max}^D$ be a global Zhou weight related to the priori
\[\mathscr{P}=\bigg(\{\mathbf{z}_i\}_{i=1}^p, \{\mathbf{f}_{0,i}\}_{i=1}^p, \{u_{0,i}\}_{i=1}^p\bigg)\]
on $D$. For any $m\in\mathbb{N}_+$, we define two compact subsets of $\mathcal{O}(D)$ (in the topology of uniform convergence on any compact subsets) as follows:
	\[\mathscr{E}_m(D):=\{f\in\mathcal{O}(D) : \sup_{z\in D}|f(z)|\le 1, (f,\mathbf{z}_i)\in\mathcal{I}(m\Phi^D_{\mathscr{P},\max})_{\mathbf{z}_i}, \ \forall i=1,\ldots,p\},\]
	\[\mathscr{A}^2_m(D):=\{f\in\mathcal{O}(D) : \|f\|_D\le 1, (f,\mathbf{z}_i)\in\mathcal{I}(m\Phi^D_{\mathscr{P},\max})_{\mathbf{z}_i}, \ \forall i=1,\ldots,p\},\] 
   where $\|f\|_D^2:=\int_D|f|^2$. Then compactness of $\mathscr{E}_m(D)$ and $\mathscr{A}^2_m(D)$ is due to Montel's theorem and Lemma \ref{lem-module}. We also define two continuous and plurisubharmonic functions $\phi_m$ and $\varphi_m$ on $D$ for any $m$ by:
   \begin{equation}\label{eq-phim}
	\phi_m(z):=\sup_{f\in\mathscr{E}_m(D)} \frac{1}{m}\log |f(z)|, \ \forall z\in D,
   \end{equation}

   \begin{equation}\label{eq-varphim}
	\varphi_m(z):=\sup_{f\in\mathscr{A}_m^2(D)}\frac{1}{m}\log|f(z)|, \ \forall z\in D.
\end{equation}
Here the continuity of $\phi_m$ and $\varphi_m$ actually means that $e^{\phi_m}$ and $e^{\varphi_m}$ are continuous.
   
   We will show that the following theorem holds.
   
   \begin{Theorem}\label{thm-approximation}
	   If $D$ is a bounded strictly hyperconvex domain (see Definition \ref{def-strhpconvex}), then

	   (1) we have
	   \[\lim_{m\to\infty}\phi_m(z)=\lim_{m\to\infty}\varphi_m(z)=\Phi^D_{\mathscr{P},\max}(z),\ \forall z\in D;\]

		(2) there exists a positive constant $\mathsf{C}$ independent of $m$, such that for any $m\in\mathbb{N}_+$,
		\[1-\frac{\mathsf{C}}{m}\le\sigma_{\mathbf{z}_i}\big(\phi_m,\Phi^D_{\mathscr{P},\max}\big)\le 1, \ \forall i\in\{1,\ldots,p\},\]
		and
		\[1-\frac{\mathsf{C}}{m}\le\sigma_{\mathbf{z}_i}\big(\varphi_m,\Phi^D_{\mathscr{P},\max}\big)\le 1,  \ \forall i\in\{1,\ldots,p\}.\]

   \end{Theorem}
   When $p=1$, this theorem was proved in \cite{BGMY23}. For the similar methods to approximate the multipoled pluricomplex Green functions and their applications, see e.g. \cite{Ni00, DT16, Ni21}.

\subsection{Convergence of global Zhou weights with distinct poles}
Let $\phi_{o,\max}$ be a local Zhou weights  the origin $o\in\mathbb{C}^n$ related to $|\mathbf{f}_{0}|^2e^{-2u_{0}}$. Denote $\tau_a$ be the translation operator: for any function $f$ defined on a domain in $\mathbb{C}^n$ taking values in $\mathbb{R}$ or $\mathbb{C}$,
   \[\tau_a f(\cdot):=f(\cdot-a).\]
   Note that for any $z_0\in\mathbb{C}^n$, $\tau_{z_0}\phi_{o,\max}$ is a local Zhou weight near $z_0$ related to $|\tau_{z_0}\mathbf{f}_{0}|^2e^{-2\tau_{z_0}u_{0}}$.

   Let $p\in\mathbb{N}_+$, and $\phi_{i,o,\max}$ be local Zhou weights near the origin $o\in\mathbb{C}^n$ related to $|\mathbf{f}_{0,i}|^2e^{-2u_{0,i}}$ for $i=1,\ldots,p$. Let $D$ be a bounded hyperconvex domain in $\mathbb{C}^n$.
   Let
    \[\mathbf{Z}:=\{\mathbf{z}_i\}_{i=1}^{p}, \ \mathbf{f}_i:=\tau_{\mathbf{z}_i}\mathbf{f}_{0,i}, \ u_i:=\tau_{\mathbf{z}_i}u_{0,i}, \ \phi_{\mathbf{z}_i,\max}:=\tau_{\mathbf{z}_i}\phi_{i,o,\max},\]
    for every $i\in\{1,\ldots,p\}$. Then for any distinct $\mathbf{z}_1,\ldots,\mathbf{z}_p\in D$, the function defined by
   \begin{flalign*}
    \begin{split}
        \Phi_{\mathbf{Z},\max}^D(z):=\sup\big\{\phi(z) :\ &\phi\in\mathrm{PSH}^-(D), \ (\mathbf{f}_{i},\mathbf{z}_i)\notin\mathcal{I}(u_{i}+\phi)_{\mathbf{z}_i},\\
        &\text{and} \ \phi\ge \phi_{\mathbf{z}_i,\max}+O(1) \ \text{near} \ \mathbf{z}_i, \ \forall i=1,\ldots,p \big\}
    \end{split}
    \end{flalign*}
    is a global Zhou weight on $D$ related to
    \[\mathscr{P}_{\mathbf{Z}}:=\bigg(\{\mathbf{z}_i\}_{i=1}^p, \{\mathbf{f}_{i}\}_{i=1}^p, \{u_{i}\}_{i=1}^p\bigg).\]
     Now we fix $\{\mathbf{f}_{0,i}\}_{i=1}^p$, $\{u_{i}\}_{i=1}^p$ and $\{\phi_{i,o,\max}\}_{i=1}^p$ to study the behavior of the global Zhou weight $\Phi_{\mathbf{Z},\max}^D$ as the pole set $\mathbf{Z}$ moves in $D$.

    Let $\mathbf{Z}^{\varepsilon}=\{\mathbf{z}_i^{\varepsilon}\}_{i=1}^p$ be a set of $p$ distinct points in $D$ for any $\varepsilon\in (0,1)$, such that $\mathbf{z}_i^{\varepsilon}\to \mathbf{z}_i$ as $\varepsilon\to 0^+$ for any $i=1,\ldots,p$. Let $\Phi_{\mathbf{Z}^{\varepsilon},\max}$ be the multipoled global Zhou weight on $D$ related to
    \[\mathscr{P}_{\mathbf{Z}^{\varepsilon}}:=\bigg(\{\mathbf{z}^{\varepsilon}_i\}_{i=1}^p, \{\mathbf{f}_{i}\}_{i=1}^p, \{u_{i}\}_{i=1}^p\bigg).\]
    We prove the following continuity property for multipoled global Zhou weights.
    \begin{Theorem}\label{thm-Zvarepsilon.to.Z}
        Suppose $\mathbf{Z}=\{\mathbf{z}_1,\ldots,\mathbf{z}_p\}$ is a set of $p$ distinct points in $D$. Then
        \[\lim_{\varepsilon\to 0^+}\Phi_{\mathbf{Z}^{\varepsilon},\max}^D=\Phi_{\mathbf{Z},\max}^D,\]
        where the convergence is pointwise on $\overline{D}$ and uniform on every compact subset of $\overline{D}\setminus\mathbf{Z}$.
    \end{Theorem}
    Theorem \ref{thm-Zvarepsilon.to.Z} shows the continuity of multipoled global Zhou weights with respect to their poles provided that the poles do not collide, i.e.,
    \[\Phi_{\mathbf{Z},\max}^D(z)=\Phi_{\{\mathbf{z}_1,\ldots,\mathbf{z}_p\},\max}^D(z)\]
    is a continuous function defined on
    \[\Big\{(z,\mathbf{z}_1,\ldots,\mathbf{z}_n)\in\overline{D}\times D^p : z\neq\mathbf{z}_i\neq\mathbf{z}_j, \ \forall i\neq j\Big\}.\]

\subsection{Semi-continuity for Zhou numbers}
   The following Siu's semi-continuity theorem is widely used in several complex variables, complex geometry, and pluripotential theory:
   \begin{Theorem}[Siu's semi-continuity theorem \cite{Siu74}, see also \cite{AMAG} Corollary 13.3]
    Let $\varphi$ be a plurisubharmonic function on a complex manifold $X$. Then, for every $c>0$, the Lelong number upper-level set
    \[E_c(\varphi)=\big\{z\in X : \nu(\varphi,z)\ge c\big\}\]
    is an analytic subset of $X$.
   \end{Theorem}

   Some important generalizations of Siu's semi-continuity theorem were given in \cite{Dem87a}, \cite{Rash06} (Lemma \ref{Lem-Rash.ana.thm} in the present paper), etc.

   We prove the corresponding semi-continuity result for Zhou numbers in the present paper.

   \begin{Theorem}\label{thm-semi.continuity}
    Let $\phi_{o,\max}$ be a local Zhou weights near the origin $o\in\mathbb{C}^n$. Let $D$ be a domain in $\mathbb{C}^n$. Then for any $\psi\in\mathrm{PSH}(D)$, and any $c>0$, the set
   \[E_c(\psi,\phi_{o,\max}):=\big\{z\in D : \sigma_{z}(\psi, \tau_{z}\phi_{o,\max})\ge c\big\}\]
   is an analytic subset of $D$.
   \end{Theorem}

\subsection{Organizations}
This paper is organized as follows. Section \ref{sec-preliminaries} gives some preliminaries, including the definitions and properties of local and global Zhou weights, while some lemmas required in the present paper are also recalled in this section. The properties of multipoled global Zhou weights appeared in Section \ref{subsec.basic.property} are proved in Section \ref{sec.basic.property}. We prove the approximation and convergence result of multipoled global Zhou weights in Section \ref{sec-approximation} and Section \ref{sec-convergence} respectively. The proof of the semi-continuity result (Theorem \ref{thm-semi.continuity}) is provided in Section \ref{sec-semicontinuity}. We write the Appendix (Section \ref{sec-appendix}) to recall the proof a Rashkovskii's analyticity result for relative types (Lemma \ref{Lem-Rash.ana.thm}), while the settings and proof are slightly different. 

\section{Preliminaries}\label{sec-preliminaries}
In this section, we review the definitions and properties of the local and global Zhou weights with one single pole in \cite{BGMY23}. Some other required lemmas are also given in this section.

\subsection{Local Zhou weights}
We first recall the definition of the local Zhou weights.

Let $f_{0}=(f_{0,1},\cdots,f_{0,m})$ be a vector,
where $f_{0,1},\cdots,f_{0,m}$ are holomorphic functions near the origin $o\in\mathbb{C}^n$. Denote $|f_{0}|^{2}=|f_{0,1}|^{2}+\cdots+|f_{0,m}|^{2}$. Let $\varphi_{0}$ be a plurisubharmonic function near $o$,
such that $|f_{0}|^{2}e^{-2\varphi_{0}}$ is integrable near $o$.

\begin{Definition}[= Definition \ref{def-max.relat.in.intro}]\label{def-max.relat}
    We call that $\phi^{f_0,\varphi_0}_{o,\max}$ ($\phi_{o,\max}$ for short) is a \textbf{local Zhou weight related to $|f_{0}|^{2}e^{-2\varphi_{0}}$ near $o$},
    if the following three statements hold:
    
    (1) $|f_{0}|^{2}e^{-2\varphi_{0}}|z|^{2N_{0}}e^{-2\phi_{o,\max}}$ is integrable near $o$ for large enough $N_{0}\gg 0$;
    
    (2) $|f_{0}|^{2}e^{-2\varphi_{0}}e^{-2\phi_{o,\max}}$
    is not integrable near $o$;
    
    (3) for any plurisubharmonic function $\varphi'\geq\phi_{o,\max}+O(1)$ near $o$ such that $|f_{0}|^{2}e^{-2\varphi_{0}}e^{-2\varphi'}$
    is not integrable near $o$, $\varphi'=\phi_{o,\max}+O(1)$ holds.
\end{Definition}

The following remark gives the existence of local Zhou weights.
	
\begin{Remark}[Remark 1.3 in \cite{BGMY23}]\label{rem-existence.single.pole}
		Let $\varphi$ be a plurisubharmonic function near $o$. Assume that
        \[|f_{0}|^{2}e^{-2\varphi_{0}}|z|^{2N_{0}}e^{-2\varphi}\]
        is integrable near $o$ for large enough $N_{0}\gg0$,
		and $(f_{0},o)\not\in\mathcal{I}(\varphi+\varphi_{0})$ holds. Then there exists a local Zhou weight $\phi_{o,\max}$ related to $|f_{0}|^{2}e^{-2\varphi_{0}}$ near $o$
		such that $\phi_{o,\max}\geq\varphi$.
		
		Moreover, $\phi_{o,\max}\geq N\log|z|+O(1)$ near $o$ for some $N\gg 0$. 
\end{Remark}

For a local Zhou weight $\phi_{o,\max}$, and any plurisubharmonic function $\psi$ near $o$, denote 
\[\sigma_o(\psi,\phi_{o,\max}):=\sup\big\{b\ge 0 : \psi\leq b \phi_{o,\max}+O(1) \ \text{near} \ o\big\},\]
which is the \emph{relative type} of $\psi$ with respect to $\phi_{o,\max}$, and called \textbf{Zhou number} for local Zhou weights particularly. The Zhou number satisfies:

\begin{Proposition}[see Remark 1.4 in \cite{BGMY23}]\label{prop-psi.le.sigma.phi.local}
    For any plurisubharmonic function $\psi$ near $o$, it holds that
    \[\psi<\sigma_o(\psi,\phi_{o,\max})\phi_{o,\max}+O(1)\]
    near $o$.
\end{Proposition}

The Zhou numbers can be computed by the following formula.

\begin{Theorem}[Theorem 1.8 in \cite{BGMY23}]\label{thm-formula.Zhou.number}
    Let $\phi_{o,\max}$ be a local Zhou weight related to $|f_{0}|^{2}e^{-2\varphi_{0}}$ near $o$. Then for any plurisubharmonic function $\psi$ near $o$,
\begin{equation*}
    \sigma_o(\psi,\phi_{o,\max})=\lim_{t\to+\infty}
\frac{\int_{\{\phi_{o,\max}<-t\}}|f_{0}|^{2}e^{-2\varphi_{0}}(-\psi)}{t\int_{\{\phi_{o,\max}<-t\}}|f_{0}|^{2}e^{-2\varphi_{0}}}.
\end{equation*}
\end{Theorem}

Theorem \ref{thm-formula.Zhou.number} yields that Zhou numbers are tropically multiplicative and tropically additive.		
	
	\begin{Corollary}[Corollary 1.9 in \cite{BGMY23}]\label{cor-tropical.multi.additi}
        Let $\psi_{1}$ and $\psi_{2}$ be plurisubharmonic functions near $o$. The following statements hold:
		
		(1) for any $c_{1}\geq 0$ and $c_{2}\geq 0$,
        \[\sigma(c_{1}\psi_{1}+c_{2}\psi_{2},\phi_{o,\max})=c_{1}\sigma(\psi_{1},\phi_{o,\max})+c_{2}\sigma(\psi_{2},\phi_{o,\max});\]
		
		(2) for holomorphic functions $f_{1}$ and $f_{2}$  near $o$,
        \[\sigma(\log|f_{1}+f_{2}|,\phi_{o,\max})\geq\min\{\sigma(\log|f_{1}|,\phi_{o,\max}),\sigma(\log|f_{2}|,\phi_{o,\max})\};\]
		
		(3)
        \[\sigma(\max\{\psi_{1},\psi_{2}\},\phi_{o,\max})=\min\{\sigma(\psi_{1},\phi_{o,\max}),\sigma(\psi_{2},\phi_{o,\max})\}.\]
	\end{Corollary}

    Especially, for any $(f,o)\in\mathcal{O}_o$, denote
\[\nu(f,\phi_{o,\max}):=\sigma(\log|f|,\phi_{o,\max}),\]
then $\nu(\cdot,\phi_{o,\max})$ is a valuation of $\mathcal{O}_o$, which is called \emph{Zhou valuation}.

\begin{Corollary}[Corollary 1.10 in \cite{BGMY23}]
	For any local Zhou weight $\phi_{o,\max}$ near $o$, the function $\nu(\cdot,\phi_{o,\max}):\mathcal{O}_o\rightarrow\mathbb{R}_{\ge0}\cup\{+\infty\}$ satisfies the following statements:
	
	$(1)$ $\nu(fg,\phi_{o,\max})=\nu(f,\phi_{o,\max})+\nu(g,\phi_{o,\max})$;
	
	$(2)$ $\nu(f+g,\phi_{o,\max})\ge\min\{\nu(f,\phi_{o,\max}),\nu(g,\phi_{o,\max})\}$;
	
	$(3)$ $\nu(f,\phi_{o,\max})=0$ for any $f(o)\not=0$.
\end{Corollary}

Let $(G,o)\in\mathcal{O}_o$. Recall the notation of the jumping number:
\[c_{o}^G(\phi_{o,\max}):=\sup\{c\ge 0 : (G,o)\in\mathcal{I}(c\phi_{o,\max})_o\},\]
and the \emph{complex singularity exponent}: $c_{o}(\phi_{o,\max}):=c_{o}^1(\phi_{o,\max})$.
\begin{Theorem}[Theorem 1.11 in \cite{BGMY23}]\label{thm-jump.no.Zhou.no.control}
    For any $(G,o)\in\mathcal{O}_o$, and local Zhou weight $\phi_{o,\max}$ related to $|f_0|^2e^{-2\varphi_0}$ near $o$,
    \begin{flalign*}
        \begin{split}
        \nu(G,\phi_{o,\max})+c_o(\phi_{o,\max})\le& c^G_o(\phi_{o,\max})\\
				\le& \nu(G,\phi_{o,\max})-\sigma(\log|f_0|,\phi_{o,\max})+\sigma(\varphi_0,\phi_{o,\max})+1.
        \end{split}
    \end{flalign*}
\end{Theorem}

\subsection{Global Zhou weights}
Now we recall the global Zhou weights.

	Let $D$ be a  domain in $\mathbb{C}^n$ with $o\in D$.	Let $f_{0}=(f_{0,1},\cdots,f_{0,m})$ be a vector,
	where $f_{0,1},\cdots,f_{0,m}$ are holomorphic functions near $o$. Denote $|f_{0}|^{2}=|f_{0,1}|^{2}+\cdots+|f_{0,m}|^{2}$. Let $\varphi_{0}$ be a plurisubharmonic function near $o$, such that $|f_{0}|^{2}e^{-2\varphi_{0}}$ is integrable near $o$. 
	\begin{Definition}[=Definition \ref{def-global.Zhou.weight.in.intro}]
		We call that a negative plurisubharmonic function $\Phi^{f_0,\varphi_0,D}_{o,\max}$ ($\Phi^{D}_{o,\max}$ for short) on $D$ a \textbf{global Zhou weight related to $|f_0|^2e^{-2\varphi_0}$ on $D$} if the following statements hold:
		
		$(1)$ $|f_0|^2e^{-2\varphi_0}|z|^{2N_0}e^{-2\Phi^{D}_{o,\max}}$ is integrable near $o$ for large enough $N_0\gg 0$;
		
		$(2)$ $|f_0|^2e^{-2\varphi_0-2\Phi^{D}_{o,\max}}$ is not integrable near $o$;
		
		$(3)$ for any negative plurisubharmonic function $\tilde\varphi$ on $D$ satisfying that $\tilde\varphi\ge\Phi^{D}_{o,\max}$ on $D$ and $|f_0|^2e^{-2\varphi_0-2\tilde\varphi}$ is not integrable near $o$, $\tilde\varphi=\Phi^{D}_{o,\max}$ holds on $D$.
	\end{Definition}

    Similar as the case of local Zhou weights, the following remark shows the existence of the global Zhou weights.

\begin{Remark}[Remark 8.2 in \cite{BGMY23}]
	Assume that there exists a negative plurisubharmonic function $\varphi$ on $D$ such that $|f_0|^2e^{-2\varphi_0-2\varphi}|z|^{2N_0}$ is integrable near $o$ for large enough $N_0\gg 0$ and $(f_0,o)\not\in\mathcal{I}(\varphi+\varphi_0)_o$.
	
	Then there exists a global Zhou weight $\Phi^{D}_{o,\max}$ on $D$ related to $|f_0|^2e^{-2\varphi_0}$ such that $\Phi^{D}_{o,\max}\ge\varphi$ on $D$.
	
\end{Remark}

We recall the definitions of \emph{hyperconvex domains} and \emph{strictly hyperconvex domains} as follows.

\begin{Definition}[see \cite{Ni95}]\label{def-strhpconvex}
	A domain $D\subset\mathbb{C}^n$ is said to be hyperconvex if there exists a continuous plurisubharmonic exhausted function $\varrho : D\to (-\infty,0)$.

	A bounded domain $D\subset\mathbb{C}^n$ is said to be strictly hyperconvex if there exists a bounded domain $\Omega$ and a function $\varrho : \Omega\to (-\infty,1)$ such that $\varrho\in C(\Omega)\cap \mathrm{PSH}(\Omega)$, $D=\{z\in\Omega : \varrho(z)<0\}$, $\varrho$ is exhaustive for $\Omega$ and for any real number $c\in [0,1]$, the open set $\{z\in\Omega : \varrho(z)<c\}$ is connected. 
\end{Definition}

It is well-known that the \emph{pluricomplex Green function} with a pole $z\in D$ on a hyperconvex domain $D\subset\mathbb{C}^n$:
\[G_D(z,\cdot):=\sup\big\{u(\cdot) : u\in\mathrm{PSH}^-(D), \ \limsup_{\zeta\to z}\big(u(\zeta)-\log|\zeta-z|\big)<+\infty\big\}\]
exists and is exhaustive on $D$. For hyperconvex domains, the following proposition gives a comparison between the global Zhou weights and the pluricomplex Green functions.
\begin{Proposition}[Lemma 8.3 in \cite{BGMY23}]
    Let $D\subset\mathbb{C}^n$ be a hyperconvex domain such that $o\in D$. For any global Zhou weight $\Phi^D_{o,\max}$ on $D$ (with the pole $o$), there exists a sufficiently large $N>0$ such that
    \[\Phi_{o,\max}^D\ge NG(z,\cdot)\]
    on $D$.
\end{Proposition}

We list other propositions of global Zhou weights as follows.

The proposition below states that one can produce a global Zhou weight by a local Zhou weight.
\begin{Proposition}[Lemma 8.5 in \cite{BGMY23}]\label{prop-local.produce.global.single.pole}
        Let $D$ be a hyperconvex domain in $\mathbb{C}^n$ containing the origin $o$, and $\phi_{o,\max}$ a local Zhou weight related to $|f_0|^2e^{-2\varphi_0}$ near $o$. Denote that
        \begin{flalign*}
            \begin{split}
                \Phi^D_{o,\max}(\cdot):=\sup\big\{u(\cdot) : u\in\mathrm{PSH}^-(D), \ (f_0,o)\notin\mathcal{I}(\varphi_0+u)_o,& \\
                u\ge \phi_{o,\max}+O(1) \ \text{near} \ o&\big\}
            \end{split}
        \end{flalign*}
    is a global Zhou weight related to $|f_0|^2e^{-2\varphi_0}$ on $D$ satisfying that $\Phi^D_{o,\max}=\phi_{o,\max}+O(1)$ near $o$.
\end{Proposition}

Conversely, globals Zhou weights are also local Zhou weights.
\begin{Proposition}[Lemma 8.7 in \cite{BGMY23}]\label{prop-global.is.local.single.pole}
    Let $D$ be a hyperconvex domain in $\mathbb{C}^n$ containing the origin $o$, and $\Phi^D_{o,\max}$ a global Zhou weight related to $|f_0|^2e^{-2\varphi_0}$ on $D$. Then $\Phi_{o,\max}$ is also a local Zhou weight related to $|f_0|^2e^{-2\varphi_0}$ near $o$.
\end{Proposition}

The following proposition is a global version of Proposition \ref{prop-psi.le.sigma.phi.local}
\begin{Proposition}[Lemma 8.8 in \cite{BGMY23}]\label{prop-psi.le.sigma.phi.single.pole}
    Let $D$ be a hyperconvex domain in $\mathbb{C}^n$ containing the origin $o$, and $\Phi^D_{o,\max}$ a global Zhou weight on $D$. Then for any $\psi\in\mathrm{PSH}^-(D)$, it holds on $D$ that
	\[\psi\le\sigma_o(\psi,\Phi^{D}_{o,\max})\Phi^{D}_{o,\max}.\]
\end{Proposition}

Proposition \ref{prop-maximal.single.pole} and Proposition \ref{prop-exhaustion.conti.single.pole} show the maximality and continuity of global Zhou weights respectively.

\begin{Proposition}[Proposition 1.18 in \cite{BGMY23}]\label{prop-maximal.single.pole}
    Let $D$ be a hyperconvex domain in $\mathbb{C}^n$ containing the origin $o$, and $\Phi^D_{o,\max}$ a global Zhou weight on $D$. Then $\Phi^{D}_{o,\max}\in L_{\mathrm{loc}}^{\infty}(D\backslash\{o\})$, and 
		\[(dd^c\Phi^{D}_{o,\max})^n=0 \ \text{on} \ D\setminus\{o\}.\]
\end{Proposition}

\begin{Proposition}[Proposition 1.19 in \cite{BGMY23}]\label{prop-exhaustion.conti.single.pole}
    Let $D$ be a bounded hyperconvex domain in $\mathbb{C}^n$ containing the origin $o$, and $\Phi^D_{o,\max}$ a global Zhou weight on $D$. Then $e^{\Phi_{o,\max}^D}$ is continuous on $D$, and $\Phi_{o,\max}^D(z)\to 0$ as $z\to\partial D$.
\end{Proposition}

\subsection{Other required lemmas.}\label{subsec-required.lemma}

First, we recall the strong openness property of multiplier ideal sheaves, which is conjectured by Demailly, and proved by Guan-Zhou.
\begin{Theorem}[\cite{GZ15a}, see also \cite{Lem14} and \cite{Hiep14}]\label{thm-SOC}
    Let $(\varphi_{j})$ be a sequence of plurisubharmonic functions increasing convergence to a plurisubharmonic function $\varphi$. Then $\mathcal{I}(\varphi)=\bigcup_{j}\mathcal{I}(\varphi_{j})$.
\end{Theorem}

In the following, we recall a generalization of the strong openness property.

Let $\{\phi_{m}\}_{m\in\mathbb{N}^{+}}$ be a sequence of negative plurisubharmonic functions on $\Delta^{n}$,
	which is convergent to a negative Lebesgue measurable function $\phi$ on $\Delta^{n}$ in Lebesgue measure.
	
	Let $f$ be a holomorphic function near $o$, and let $I$ be an ideal of $\mathcal{O}_{o}$.
	We denote 
	\[C_{f,I}(U):=\inf\left\{\int_{U} |\tilde{f}|^{2} :(\tilde{f}-f,o)\in I \ \& \ \tilde{f}\in\mathcal{O}(U)\right\},\]
	where $U\subseteq \Delta^{n}$ is a domain with $o\in U$.
	Especially, if $I=\mathcal{I}(\phi)_{o}$, we denote $C_{f,\phi}(U):=C_{f,I}(U)$.
	
	In \cite{GZ15b}, Guan-Zhou presented the following lower semicontinuity property for
	plurisubharmonic functions with a multiplier.
	
	\begin{Lemma}[\cite{GZ15b}]\label{lem-effect_GZ}
		Let $f_0$ be a holomorphic function near $o$.
		Assume that for any small enough neighborhood $U$ of $o$,
		the pairs $(f_0,\phi_{m})$ $(m\in\mathbb{N}^+)$ satisfies
		\begin{equation}
			\label{eq-C.f.phi.U}
			\inf_{m}C_{f_0, 2\varphi_0+\phi_{m}}(U)>0.
		\end{equation}
		Then $|f_{0}|^{2}e^{-2\varphi_{0}}e^{-\phi}$ is not integrable near $o$.
	\end{Lemma}
	
	The Noetherian property of multiplier ideal sheaves (see \cite{CADG}) shows that
	\begin{Remark}\label{rem-effect_GZ}
		Assume that
		
		(1) $\phi_{m+1}\geq\phi_{m}$ holds for any $m$;
		
		(2) $|f_{0}|^{2}e^{-2\varphi_{0}}e^{-\phi_{m}}$ is not integrable near $o$ for any $m$.
		
		Then inequality \eqref{eq-C.f.phi.U} holds.
	\end{Remark}

    The lemma below is also needed.
\begin{Lemma}[see \cite{AMAG}, or \cite{BGMY23} Lemma 3.1]\label{lem-integrable.to.integrable}
    Let $u$ and $v$ be Lebesgue measurable functions with local upper-bound near $o$. Let $g$ be a nonnegative Lebesgue measurable function near $o$. Assume that $g^{2}e^{2(l_{1}v-(1+l_{2})u)}$ is integrable near $o$, where $l_{1},l_{2}>0$. Then
    $g^{2}e^{-2u}-g^{2}e^{-2\max\big\{u,\frac{l_{1}}{l_{2}}v\big\}}$ is also integrable near $o$.
\end{Lemma}

Lemma \ref{lem-Choquet} and Lemma \ref{lem-upper.reg.equal.a.e} give the results about upper regularizations of a family of upper-semicontinuous functions.

\begin{Lemma}[\emph{Choquet's lemma}, see \cite{CADG}]\label{lem-Choquet}
    Every family $(u_{\alpha})$ of uppersemicontinuous functions has a countable subfamily $(v_{j})=(u_{\alpha(j)})$,
    such that its upper envelope $v=\sup_{j}v_{j}$ satisfies $v\leq u\leq u^{*} = v^{*}$,
    where
    $u=\sup_{\alpha}u_{\alpha}$, and
    \[u^{*}(z):=\lim_{\varepsilon\to0}\sup_{\mathbb{B}^{n}(z,\varepsilon)}u, \ v^{*}(z):=\lim_{\varepsilon\to0}\sup_{\mathbb{B}^{n}(z,\varepsilon)}v\]
    are the regularizations of $u$ and $v$.
\end{Lemma}

\begin{Lemma}[see Proposition 4.24 in \cite{CADG}]\label{lem-upper.reg.equal.a.e}
    If all $(u_{\alpha})$ are subharmonic, the upper regularization $u^*$ is subharmonic and equals almost everywhere to $u$.
\end{Lemma}

The following lemma shows the closedness of the ideals of $\mathcal{O}_o$.
\begin{Lemma}[see \cite{GR84}]\label{lem-module}
	Let $I$ be an ideal of $\mathcal{O}_{\mathbb{C}^n,o}$. Let $\{f_j\}\subset\mathcal{O}_{\mathbb{C}^n}(U)$ be a sequence of holomorphic functions in an open neighborhood $U$ of the origin $o$. Assume that $\{f_j\}$ converges uniformly in $U$ towards $f\in\mathcal{O}_{\mathbb{C}^n,o}$, and assume furthermore that all the germs $(f_j,o)$ belong to $I$. Then $(f,o)\in I$.
\end{Lemma}

The following lemma is called \emph{Demailly's approximation theorem}, which gives the approximation of any plurisubharmonic function by using the Bergman kernels.
\begin{Lemma}[see \cite{AMAG}]\label{lem-Dem.approx}
	Let $D\subset\mathbb{C}^n$ be a bounded pseudoconvex domain, and $\varphi\in \mathrm{PSH}(D)$. For any positive integer $m$, let $\{\sigma_{m,k}\}_{k=1}^{\infty}$ be an orthonormal basis of
    \[A^2(D,2m\varphi):=\left\{f\in\mathcal{O}(D):\int_{D}|f|^2e^{-2m\varphi}d\lambda<+\infty\right\}.\]
    Denote
	\[\varphi_m:=\frac{1}{2m}\log\sum_{k=1}^{\infty}|\sigma_{m,k}|^2\]
	on $D$. Then there exist positive constants $C_1$ (depending only on $n$ and diameter of $D$) and $C_2$ such that 
	\begin{equation*}
		\varphi(z)-\frac{C_1}{m}\le\varphi_m(z)\le\sup_{|\zeta-z|<r}\varphi(\zeta)+\frac{1}{m}\log\frac{C_2}{r^n}
	\end{equation*}
	for any $z\in D$ satisfying $\{\zeta : |\zeta-z|<r\}\subset\subset D$. Especially, $\varphi_m$ converges to $\varphi$ pointwisely and in $L^1_{\text{loc}}$ on $D$.
\end{Lemma}

\begin{Remark}
	Let $(\tau_l)$ be an orthonormal basis of the space $A^2(\Omega, 2m\varphi)$, then
	\[\sum_l|\tau_l(z)|^2=\sup\left\{|f(z)|^2 : f\in A^2(\Omega, 2m\varphi) \ \& \ \int_{\Omega}|f|^2e^{-2m\varphi}\le 1\right\}\]
	for any $z\in \Omega$.
\end{Remark}

Lemma \ref{lem-maximal.ddcn=0} shows the condition $(dd^c\varphi)^n=0$ for a plurisubharmonic funtion $\varphi$ actually is equivalent to a maximal property to $\varphi$.
\begin{Lemma}[see \cite{Bl-note}, see also \cite{BT76,Bl93}]\label{lem-maximal.ddcn=0}
    Let $\varphi\in \mathrm{PSH}(\Omega)\cap L^{\infty}_{\mathrm{loc}}(\Omega)$ on an open subset $\Omega$ of $\mathbb{C}^n$. Then for any $u\in \mathrm{PSH}(\Omega)$ such that $\varphi\ge u$ outside a compact subset of $\Omega$ we have $\varphi\ge u$ on $\Omega$, if and only if $(dd^c\varphi)^n=0$ on $\Omega$.
\end{Lemma}

\section{Basic properties of multipoled global Zhou weights}\label{sec.basic.property}
In this section, we prove the basic properties of multipoled global Zhou weights, including Remark \ref{rem-existence}, Proposition \ref{prop-global.is.local}, \ref{prop-local.produce.global}, \ref{prop-psi.le.sigmaZ.Phi}, \ref{prop-maximal}, and \ref{prop-continuous.exhaustive}.

\begin{proof}[\textbf{Proof of Remark \ref{rem-existence}}.]
    First, we prove the existence of $\Phi_{\mathscr{P},\max}^D$. Define a set
    \[\mathcal{S}(\mathscr{P},\Phi):=\big\{\phi\in\mathrm{PSH}^-(D) : \phi\ge \Phi, \ \text{and} \ (\mathbf{f}_{0,i},\mathbf{z}_i)\notin \mathcal{I}(u_{0,i}+\phi)_{\mathbf{z}_i}, \ \forall i=1,\ldots,p\big\}.\]
    It follows from \emph{Zorn's Lemma} that $\mathcal{S}(\mathscr{P},\Phi)$ contains a maximal subset $(\phi_{\alpha})_{\alpha\in\Gamma}\subset\mathcal{S}(\mathscr{P},\Phi)$ with the index set $\Gamma$ satisfying: for every $\alpha,\alpha'\in\Gamma$, either $\phi_{\alpha}\le\phi_{\alpha'}$ or $\phi_{\alpha}\ge\phi_{\alpha'}$ holds on $D$.

    Let $\varphi(z):=\sup_{\alpha\in\Gamma}\phi_{\alpha}(z)$, and
    \[\varphi^*(z):=\lim_{\varepsilon\to 0^+}\sup_{w\in\mathbb{B}^n(z,\varepsilon)}\varphi(w),\]
    for every $z\in D$. Then $\varphi^*\in\mathrm{PSH}^-(D)$. Choquet's Lemma (Lemma \ref{lem-Choquet}) shows that there exists an increasing subsequence $(\phi_j)_{j\ge 1}$ of $(\phi_{\alpha})_{\alpha\in\Gamma}$ such that
    \[\big(\lim_{j\to\infty}\phi_j\big)^*=\varphi^*.\]
    Thus, according to Lemma \ref{lem-upper.reg.equal.a.e}, we have also that $(\phi_j)$ converges to $u^*$ almost everywhere in $D$, where $u^*$ is plurisubharmonic on $D$. Then Lemma \ref{lem-effect_GZ} and Remark \ref{rem-effect_GZ} imply that
    \[(\mathbf{f}_{0,i},\mathbf{z}_i)\not\in \mathcal{I}(u_{0,i}+\varphi^*)_{\mathbf{z}_i}, \ \forall i=1,\ldots,p.\]
    Now by the definition of $\varphi^*$, we have that $\Phi_{\mathscr{P},\max}^D$ is a global Zhou weight related to $\mathscr{P}$ on $D$.

    Next, we prove the inequality (\ref{eq-Phi.ge.NG_D}). Since
    \[|\mathbf{f}_{0,i}|^2e^{-2u_{0,i}}|z-\mathbf{z}_i|^{2N_i}e^{-2\Phi^D_{\mathscr{P},\max}}\]
    is integrable near $\mathbf{z}_i$, the strong openness property (Theorem \ref{thm-SOC}) shows that there exists sufficiently small $\varepsilon_i>0$ such that
    \[|\mathbf{f}_{0,i}|^2e^{-2u_{0,i}}|z-\mathbf{z}_i|^{2N_i}e^{-2(1+\varepsilon_i)\Phi^D_{\mathscr{P},\max}}\]
    is integrable near $\mathbf{z}_i$ for every $i\in\{1,\ldots,p\}$. Then for the multipoled pluricomplex Green function $G_{D,\mathbf{Z}}$ on the hyperconvex domain $D$, we get that
    \[|\mathbf{f}_{0,i}|^2e^{-2u_{0,i}}e^{2N_iG_{D,\mathbf{Z}}}e^{-2(1+\varepsilon_i)\Phi^D_{\mathscr{P},\max}}\]
    is integrable near $\mathbf{z}_i$ for every $i$. Now by Lemma \ref{lem-integrable.to.integrable}, it holds that
    \[|\mathbf{f}_{0.i}|^2e^{-2u_{0,i}}e^{-2\Phi_{\mathscr{P},\max}^D}-|\mathbf{f}_{0.i}|^2e^{-2u_{0,i}}e^{-2\max \left\{\Phi_{\mathscr{P},\max}^D,\frac{N_i}{\varepsilon_i}G_{D,\mathbf{Z}}\right\}}\]
    is integrable near $\mathbf{z}_i$ for every $i$. Since
    \[|\mathbf{f}_{0.i}|^2e^{-2u_{0,i}}e^{-2\Phi_{\mathscr{P},\max}^D}\]
    is not integrable near each $\mathbf{z}_i$, it follows that 
    \[|\mathbf{f}_{0.i}|^2e^{-2u_{0,i}}e^{-2\max \left\{\Phi_{\mathscr{P},\max}^D,\frac{N_i}{\varepsilon_i}G_{D,\mathbf{Z}}\right\}}\]
    is not integrable near each $\mathbf{z}_i$. Denote
    \[N:=\max_{1\le i\le p}\{N_i/\varepsilon_i\}.\]
    Note that $NG_{D,\mathbf{Z}}\in \mathrm{PSH}^-(D)$, then by the definition of $\Phi_{\mathscr{P},\max}^D$, it holds that
    \[\Phi_{\mathscr{P},\max}^D(z)\ge NG_{D,\mathbf{Z}}(z), \ \forall z\in D.\]
\end{proof}

For any $z\in D$ and any $i\in\{1,\ldots,p\}$, denote
\begin{flalign}\label{eq-Phi.zi.z}
\begin{split}
    \Phi_{D,\mathbf{z}_i}(z)=\sup\big\{\phi(z) : \phi\in\mathrm{PSH}^-(D), \ (\mathbf{f}_{0,i},\mathbf{z}_i)\notin\mathcal{I}(u_{0,i}+\phi)_{\mathbf{z}_i},&\\
    \text{and} \ \phi\ge \Phi_{\mathscr{P},\max}^D \ \text{on} \  D&\big\}.
\end{split}
\end{flalign}
We need the following lemma.

\begin{Lemma}\label{lem-sigma.le.PhiP.le.min}
     The function $\Phi_{D,\mathbf{z}_i}(z)$ defined by (\ref{eq-Phi.zi.z}) is a global Zhou weight related to $|\mathbf{f}_{0,i}|^2e^{-2u_{0,i}}$ at $\mathbf{z}_i$ on $D$ for any $i\in\{1,\ldots,p\}$. Moreover, we have that
     \begin{equation}\label{eq-sumPhi.le.Phimax.le.min}
        \sum_{1\le i\le p}\Phi_{D,\mathbf{z}_i}(z)\le \Phi_{\mathscr{P},\max}^D(z)\le \min_{1\le i\le p}\Phi_{D,\mathbf{z}_i}(z)
     \end{equation}
     holds for every $z\in D$, and in particular,
     \[\Phi_{\mathscr{P},\max}^D=\Phi_{D,\mathbf{z}_i}+O(1)\]
     near $\mathbf{z}_i$ for any $i$.
\end{Lemma}

\begin{proof}
    Replacing $\mathbf{Z}$ by $\mathbf{z}_i$ for some fixed $i$, Remark \ref{rem-existence} shows that there exists a global Zhou weight $\Psi_{D,\mathbf{z}_i}\in \mathcal{L}_i\big(\Phi_{\mathscr{P},\max}^D\big)$ related to $|\mathbf{f}_{0,i}|^2e^{-2u_{0,i}}$ at $\mathbf{z}_i$ on $D$ for every $i$, where
    \begin{flalign*}
        \begin{split}
            \mathcal{L}_i\big(\Phi_{\mathscr{P},\max}^D\big):=\big\{\phi\in\mathrm{PSH}^-(D): \ &(\mathbf{f}_{0,i},\mathbf{z}_i)\notin\mathcal{I}(u_{0,i}+\phi)_{\mathbf{z}_i},\\
            &\text{and} \ \phi\ge \Phi_{\mathscr{P},\max}^D \ \text{on} \ D \big\}.
        \end{split}
        \end{flalign*}
        By the definition of $\Phi_{D,\mathbf{z}_i}$, we have $\Phi_{D,\mathbf{z}_i}\ge\Psi_{D,\mathbf{z}_i}$. 
        
        We next prove
        \begin{equation}\label{eq-sumPsi.le.Phimax.le.min}
            \sum_{1\le i\le p}\Psi_{D,\mathbf{z}_i}\le \Phi_{\mathscr{P},\max}^D\le \min_{1\le i\le p}\Psi_{D,\mathbf{z}_i},
         \end{equation}
        and in particular,
        \[\Psi_{D,\mathbf{z}_i}=\Phi_{\mathscr{P},\max}^D+O(1)\ \text{near} \ \mathbf{z}_i.\]
        In fact, note that the function
        \[\varPsi:=\max\bigg\{\sum_{1\le i\le p}\Psi_{D,\mathbf{z}_i}, \Phi_{\mathscr{P},\max}^D\bigg\}\in\mathrm{PSH}^-(D)\]
        satisfies
        \[\Phi_{\mathscr{P},\max}^D\le\varPsi\le \Psi_{D,\mathbf{z}_i} \ \text{on} \ D,\]
         which implies that $(\mathbf{f}_{0,i},\mathbf{z}_i)\notin\mathcal{I}(u_{0,i}+\varPsi)_{\mathbf{z}_i}$ is not integrable near $\mathbf{z}_i$ for any $i\in\{1,\ldots,p\}$. Then it follows from the definition of $\Phi_{\mathscr{P},\max}^D$ that $\varPsi=\Phi_{\mathscr{P},\max}^D$ on $D$. Thus, we deduce that
         \begin{equation*}
            \sum_{1\le i\le p}\Psi_{D,\mathbf{z}_i}\le \Phi_{\mathscr{P},\max}^D.
         \end{equation*}
         The equality $\Phi_{\mathscr{P},\max}^D\le \min\limits_{1\le i\le p}\Psi_{D,\mathbf{z}_i}$ holds by the definitions of $\Psi_{D,\mathbf{z}_i}$.
         
         Now for every $\phi\in \mathcal{L}_i\big(\Phi_{\mathscr{P},\max}^D\big)$, we have that $\phi\ge \Psi_{D,\mathbf{z}_i}+O(1)$ near $\mathbf{z}_i$, since $\Psi_{D,\mathbf{z}_i}\in L^{\infty}_{\mathrm{loc}}(D\setminus\{\mathbf{z}_i\})$ (Proposition \ref{prop-maximal.single.pole}). As $\Psi_{D,\mathbf{z}_i}$ is a global Zhou weight (also a local Zhou weight by Proposition \ref{prop-global.is.local.single.pole}) on $D$ related to $|\mathbf{f}_{0,i}|^2e^{-2u_{0,i}}$, we get $\phi=\Psi_{D,\mathbf{z}_i}+O(1)$ near $\mathbf{z}_i$, yielding that $\sigma_{\mathbf{z}_i}(\phi,\Psi_{D,\mathbf{z}_i})=1$, and $\phi\le\Psi_{D,\mathbf{z}_i}$ on $D$ (Proposition \ref{prop-psi.le.sigma.phi.single.pole}). Consequently, we have $\Phi_{D,\mathbf{z}_i}=\Psi_{D,\mathbf{z}_i}$ on $D$ for every $i\in\{1,\ldots,p\}$. In addition, the inequality (\ref{eq-sumPhi.le.Phimax.le.min}) follows from (\ref{eq-sumPsi.le.Phimax.le.min}).
\end{proof}

Now we continue our proofs.

\begin{proof}[\textbf{Proof of Proposition \ref{prop-global.is.local}}]
    For the global Zhou weight $\Phi_{\mathscr{P},\max}^D$, let $\Phi_{D,\mathbf{z}_i}$ be defined by (\ref{eq-Phi.zi.z}) for any $i=1,\ldots,p$. Then $\Phi_{D,\mathbf{z}_i}$ is a global Zhou weight related to $|\mathbf{f}_{0,i}|^2e^{-2u_{0,i}}$ at $\mathbf{z}_i$ on $D$, and
    \begin{equation*}
        \Phi_{\mathscr{P},\max}^D=\Phi_{D,\mathbf{z}_i}+O(1) \ \text{near} \ \mathbf{z}_i
    \end{equation*}
    for any $i$. According to Proposition \ref{prop-global.is.local.single.pole}, $\Phi_{D,\mathbf{z}_i}$ is also a local Zhou weight related to $|\mathbf{f}_{0,i}|^2e^{-2u_{0,i}}$ at $\mathbf{z}_i$, thus so is $\Phi_{\mathscr{P},\max}^D$ for any $i$.
\end{proof}

\begin{proof}[\textbf{Proof of Proposition \ref{prop-local.produce.global}}]
    For any fixed $i\in\{1,\ldots,p\}$, since $\phi_{\mathbf{z}_i,\max}$ is a local Zhou weight related to $|\mathbf{f}_{0,i}|^2e^{-2u_{0,i}}$ at $\mathbf{z}_i$, the function defined by
    \begin{flalign*}
        \begin{split}
            \Psi_{\mathbf{z}_i}^D(z):=\sup\big\{\phi(z) :\ \phi\in\mathrm{PSH}^-(D), \ (\mathbf{f}_{0,i},\mathbf{z}_i)\notin\mathcal{I}(u_{0,i}+\phi)_{\mathbf{z}_i}&,\\
            \text{and} \ \phi\ge \phi_{\mathbf{z}_i,\max}+O(1) \ \text{near} \ \mathbf{z}_i &\big\}
        \end{split}
        \end{flalign*}
    is a global Zhou weight related to $|\mathbf{f}_{0,i}|^2e^{-2u_{0,i}}$ at $\mathbf{z}_i$ on $D$ by Proposition \ref{prop-local.produce.global.single.pole}. Moreover, we have $\Psi_{\mathbf{z}_i}^D=\phi_{\mathbf{z}_i,\max}+O(1)$ near $\mathbf{z}_i$, and $\Psi_{\mathbf{z}_i}^D\in L^{\infty}_{\mathrm{loc}}(D\setminus\{\mathbf{z}_i\})$ for any $i$. Let
    \[\Phi:=\sum_{1\le i\le p}\Psi_{\mathbf{z}_i}^D \in\mathrm{PSH}^-(D).\]
    Then $\Phi=\phi_{\mathbf{z}_i,\max}+O(1)$ near $\mathbf{z}_i$ for any $i$. Thus, according to that $\phi_{\mathbf{z}_i,\max}$ are local maximal Zhou weights and Remark \ref{rem-existence}, there exists a global Zhou weight, denoted by $\Psi_{\mathscr{P},\max}^D\in\mathrm{PSH}^-(D)$ on $D$ related to
    \[\mathscr{P}=\bigg(\{\mathbf{z}_i\}_{i=1}^p, \{\mathbf{f}_{0,i}\}_{i=1}^p, \{u_{0,i}\}_{i=1}^p\bigg),\]
    such that $\Psi_{\mathscr{P},\max}^D\ge \Phi$. In particular, $\Psi_{\mathscr{P},\max}^D\ge \phi_{\mathbf{z}_i,\max}+O(1)$ near $\mathbf{z}_i$ for any $i$.

    Denote
    \begin{flalign*}
        \begin{split}
            \mathcal{L}:=\big\{\phi\in\mathrm{PSH}^-(D) : & \ (\mathbf{f}_{0,i},\mathbf{z}_i)\notin\mathcal{I}(u_{0,i}+\phi)_{\mathbf{z}_i}, \\
            &\text{and} \ \phi\ge \phi_{\mathbf{z}_i,\max}+O(1) \ \text{near} \ \mathbf{z}_i, \ \forall i=1,\ldots,p \big\}.
        \end{split}
        \end{flalign*}
    Then $\Psi_{\mathscr{P},\max}^D\in\mathcal{L}$, thus $\Psi_{\mathscr{P},\max}^D\le\Phi_{\mathscr{P},\max}^D$ on $D$. On the other hand, for any $\phi\in\mathcal{L}$, let
    \[\tilde{\phi}:=\max\big\{\Psi^D_{\mathscr{P},\max}, \phi\big\}\in\mathrm{PSH}^-(D).\]
    Since $\Psi^D_{\mathscr{P},\max}\le \Psi_{\mathbf{z}_i}^D$ and $\phi\le \Psi_{\mathbf{z}_i}^D$ on $D$ by the definition of $\Psi_{\mathbf{z}_i}^D$, we have $\tilde{\phi}\le \Psi_{\mathbf{z}_i}^D$ on $D$ for any $i$, yielding that $(\mathbf{f}_{0,i},\mathbf{z}_i)\notin\mathcal{I}(u_{0,i}+\tilde{\phi})_{\mathbf{z}_i}$ for any $i$. Now according to $\tilde{\phi}\ge\Psi^D_{\mathscr{P},\max}$ and that $\Psi_{\mathscr{P},\max}^D$ is a global Zhou weight on $D$, we get $\tilde{\phi}=\phi$, thus $\phi\le\Psi_{\mathscr{P},\max}^D$ and then $\Psi_{\mathscr{P},\max}^D\ge\Phi_{\mathscr{P},\max}^D$ on $D$. Consequently, $\Phi_{\mathscr{P},\max}^D=\Psi_{\mathscr{P},\max}^D$ is a global Zhou weight on $D$ related to $\mathscr{P}$.

    For every $i\in\{1,\ldots,p\}$, note that $\Phi_{\mathscr{P},\max}^D\ge \phi_{\mathbf{z}_i,\max}+O(1)$ near $\mathbf{z}_i$, and $(\mathbf{f}_{0,i},\mathbf{z}_i)\notin\mathcal{I}(u_{0,i}+\Phi_{\mathscr{P},\max}^D)_{\mathbf{z}_i}$, then we have $\Phi_{\mathscr{P},\max}^D= \phi_{\mathbf{z}_i,\max}+O(1)$ near $\mathbf{z}_i$ since $\phi_{\mathbf{z}_i,\max}$ is a local Zhou weight near $\mathbf{z}_i$ related to $|\mathbf{f}_{0,i}|^2e^{-2u_{0,i}}$.
\end{proof}

\begin{proof}[\textbf{Proof of Proposition \ref{prop-psi.le.sigmaZ.Phi}}]
If $\sigma_{\mathbf{Z}}\big(\psi,\Phi_{\mathscr{P},\max}^D\big)=0$, the proof is done. Otherwise, we denote
\[\Psi:=\max\left\{\frac{1}{\sigma_{\mathbf{Z}}\big(\psi,\Phi_{\mathscr{P},\max}^D\big)}\psi, \Phi_{\mathscr{P},\max}^D\right\}\in\mathrm{PSH}^-(D).\]
Clearly $\Psi\ge \Phi_{\mathscr{P},\max}^D$ on $D$. Besides, since $\Phi_{\mathscr{P},\max}^D$ is a local Zhou weight (Proposition \ref{prop-global.is.local}), we have that $\Psi\le\Phi_{\mathscr{P},\max}^D+O(1)$ near every $\mathbf{z}_i$. Thus, $(\mathbf{f}_{0,i},\mathbf{z}_i)\notin\mathcal{I}(u_{0,i}+\Psi)_{\mathbf{z}_i}$ for any $i$. Then by the definition of global Zhou weight, $\Psi=\Phi_{\mathscr{P},\max}^D$ holds on $D$, which completes the proof.
\end{proof}

\begin{proof}[\textbf{Proof of Proposition \ref{prop-maximal}}]
    Let $\Phi_{D,\mathbf{z}_i}$ be defined as which in Lemma \ref{lem-sigma.le.PhiP.le.min} for $i=1,\ldots,p$. We have $\Phi_{D,\mathbf{z}_i}\in L_{\mathrm{loc}}^{\infty}\big(D\setminus\{\mathbf{z}_i\}\big)$ for each $i$ by Proposition \ref{prop-maximal.single.pole}. Then it follows from equation (\ref{eq-sumPhi.le.Phimax.le.min}) in Lemma \ref{lem-sigma.le.PhiP.le.min} that $\Phi_{\mathscr{P},\max}^D\in L^{\infty}_{\mathrm{loc}}\big(D\setminus\{\mathbf{z}_1,\ldots,\mathbf{z}_n\}\big)$.
    
    Using Lemma \ref{lem-maximal.ddcn=0}, we get $\big(dd^c\Phi^{D}_{\mathscr{P},\max}\big)^n=0$ on $D\setminus\{\mathbf{z}_1,\ldots,\mathbf{z}_n\}$.
\end{proof}

\begin{proof}[\textbf{Proof of Proposition \ref{prop-continuous.exhaustive}}]
    First, the exhaustion of $\Phi_{\mathscr{P},\max}^D$ is a direct consequece of Lemma \ref{lem-sigma.le.PhiP.le.min} and Proposition \ref{prop-exhaustion.conti.single.pole}. As a consequece, we can define the restriction of $\Phi_{\mathscr{P},\max}^D$ on $\partial D$ to be $0$.

Next, we prove the continuity of $e^{\Phi_{\mathscr{P},\max}}$ on $D$. We apply Demailly's approximation theorem (Lemma \ref{lem-Dem.approx}) to $\Phi_{\mathscr{P},\max}^D$. Let
\[A^2(D,2m\Phi_{\mathscr{P},\max}^D):=\left\{f\in\mathcal{O}(D) : \int_D|f|^2e^{-2m\Phi_{\mathscr{P},\max}^D}<+\infty\right\}.\]
Then the function defined by
\[\Phi_m(z):=\sup_{f\in A^2(D,2m\Phi_{\mathscr{P},\max}^D)}\frac{1}{2m}\log|f(z)|, \ z\in D,\]
satisfies
\begin{equation}\label{eq-Phim.ge.PhiP-C/m}
    \Phi_m\ge \Phi_{\mathscr{P},\max}^D-\frac{C_1}{m}
\end{equation}
on $D$ for any $m\in\mathbb{N}_+$, where $C_1$ is a constant independent of $m$.

Since $D$ is a bounded hyperconvex domain, there exists $\varrho\in C(D)\cap\mathrm{PSH}(D)$ which is exhaustive on $D$. Let $\epsilon>0$ be sufficiently small, and denote
\[D_{\epsilon}:=\big\{z\in D : \varrho(z)<-\epsilon\big\}.\]
Let $\Phi_{\mathscr{P},\max}^{D_{\epsilon}}$ be the (unique) global Zhou weight on $D_{\epsilon}$ related to $\mathscr{P}$ with
$\Phi_{\mathscr{P},\max}^{D_{\epsilon}}=\Phi_{\mathscr{P},\max}^{D}+O(1)$ near $\mathbf{z}_i$, $\forall i=1,\ldots,p$. Let
\[\delta_{\varepsilon}:=\inf_{z\in D\setminus\overline{D_{2\epsilon}}}\Phi_{\mathscr{P},\max}^D(z)\in (-\infty,0).\]
Then $\Phi_{\mathscr{P},\max}^{D}-\delta_{\varepsilon}\ge \Phi_{\mathscr{P},\max}^{D_{\epsilon}}$ on $\overline{D_{\epsilon}}\setminus\overline{D_{2\epsilon}}$. Denote
\begin{equation*}
    \Phi_{\epsilon}(z):=\left\{
    \begin{array}{ll}
        \max\left\{\Phi_{\mathscr{P},\max}^{D_{\epsilon}}(z)+\delta_{\epsilon}, \Phi_{\mathscr{P},\max}^D(z)\right\} & z\in D_{\epsilon}, \\
        \Phi_{\mathscr{P},\max}^{D}(z) & z\in D\setminus D_{\epsilon}.
    \end{array}
    \right.
\end{equation*}
Then $\Phi_{\epsilon}\in \mathrm{PSH}^-(D)$ and $\Phi_{\epsilon}=\Phi_{\mathscr{P},\max}^{D}+O(1)$ near $\mathbf{z}_i$, $\forall i=1,\ldots,p$. The definition of $\Phi_{\mathscr{P},\max}^{D}$ implies
\[\Phi_{\mathscr{P},\max}^{D_{\epsilon}}+\delta_{\epsilon}\le\Phi_{\mathscr{P},\max}^{D} \ \text{on} \ D_{\epsilon}.\]
Addtionally, we have $\delta_{\epsilon}\to 0$ as $\epsilon\to 0^+$ according to the exhaustion of $\Phi_{\mathscr{P},\max}^{D}(z)$.

For any $z\in D_{\epsilon}$ and any $m\in\mathbb{N}_+$, there exists $F_{z,m}\in \mathcal{O}(D)$ such that
\[\int_{D}|F_{z,m}|^2e^{-2m\Phi_{\mathscr{P},\max}}=1, \ \text{and} \ \frac{1}{m}\log|F_{z,m}(z)|=\Phi_m(z).\]
In particular,    
\[(F_{z,m},\mathbf{z}_i)\in\mathcal{I}\big(m\Phi_{\mathscr{P},\max}^D\big)_{\mathbf{z}_i}, \ \forall i=1,\ldots,p,\]
yielding that $c_{\mathbf{z}_i}^{F_{z,m}}(\Phi_{\mathscr{P},\max}^D)\ge m$. According to Theorem \ref{thm-jump.no.Zhou.no.control}, we have
\[\sigma_{\mathbf{z}_i}\big(\log|F_{z,m}|,\Phi_{\mathscr{P},\max}^D\big)\ge m-C_2, \ \forall i=1,\ldots,p,\]
where $C_2$ is a constant independent of $m$, $z$ and $F_{z,m}$. We also have
\[|F_{z,m}(\zeta)|^2\le \frac{C^2_3}{R_{\epsilon}^{2n}}\int_D|F_{z,m}|^2\mathrm{d}\lambda\le\frac{C^2_3}{R_{\epsilon}^{2n}},\ \forall \zeta\in D_{\epsilon},\]
where $C_3$ is a positive constant only dependent on $n$, and $R_{\epsilon}:=\mathrm{dist}\big(\overline{D_{\epsilon}}, \partial D\big)$. $R_{\epsilon}>0$ is a consequece of the exhaustion of $\varrho$. Thus, it follows Proposition \ref{prop-psi.le.sigmaZ.Phi} that
\[\log|F_{z,m}|-\log\frac{C_3}{R_{\epsilon}^n}\le (m-C_2)\Phi_{\mathscr{P},\max}^{D_{\epsilon}}\]
on $D_{\epsilon}$. Especially, for $m>C_2$,
\begin{flalign*}
    \begin{split}
        \Phi_m(z)&\le \frac{1}{m}\log|F_{z,m}(z)|\\
        &\le \frac{m-C_2}{m}\Phi_{\mathscr{P},\max}^{D_{\epsilon}}(z)+\frac{1}{m}\log\frac{C_3}{R_{\epsilon}^n}\\
        &\le\frac{m-C_2}{m}\big(\Phi_{\mathscr{P},\max}^{D}(z)-\delta_{\epsilon}\big)+\frac{1}{m}\log\frac{C_3}{R_{\epsilon}^n}
    \end{split}
\end{flalign*}
for any $z\in D_{\epsilon}$. Combining with equation (\ref{eq-Phim.ge.PhiP-C/m}), we get
\begin{equation}\label{eq-Phim.PhiP}
    \frac{m}{m-C_2}\left(\Phi_m-\frac{1}{m}\log\frac{C_3}{R_{\epsilon}^n}\right)+\delta_{\epsilon}\le\Phi_{\mathscr{P},\max}^D-\frac{C_1}{m}\le\Phi_m
\end{equation}
 on $D_{\epsilon}$. Note that $e^{\Phi_m}$ is continuous on $D$, and $\lim_{\epsilon\to 0^+}\delta_{\epsilon}=0$. We can deduce that $e^{\Phi_{\mathscr{P},\max}^D}$ is continuous on $D$.

\end{proof}

\section{Approximations of global Zhou weights with multipoles: proof of Theorem \ref{thm-approximation}}\label{sec-approximation}

The proof of Theorem \ref{thm-approximation} is given in this section, which is a reformation of the proof of Theorem 1.20 in \cite{BGMY23} (see also \cite{Ni95, Ni00,Ni21}).

In this section, we always assume that $D$ is a strictly hyperconvex domain with the continuous plurisubharmonic function $\varrho : \Omega \to (-\infty, 1)$ defined in Definition \ref{def-strhpconvex}. For any $j\in\mathbb{Z}\setminus\{0\}$, denote
\[D_j:=\{z\in \Omega : \varrho(z)<1/j\}.\]
Then $\{D_j\}_{j=1}^{+\infty}$ is a decreasing sequence of bounded hyperconvex domains, and $\{D_{-j}\}_{j=j_0}^{+\infty}$ is an increasing sequence of bounded hyperconvex domains. Here $j_0$ is a sufficiently large positive integer.

For $j>0$ or $j\le -j_0$, define
\begin{flalign*}
    \begin{split}
        \Phi^{D_j}_{\mathscr{P},\max}(z):=\sup\big\{\phi(z): & \ \phi\in \mathrm{PSH}^-(D_j), \ (\mathbf{f}_{0,i},\mathbf{z}_i)\notin \mathcal{I}(u_{0,i}+\phi)_{\mathbf{z}_i}, \\
        &\text{and} \ \phi\ge \Phi^D_{\mathscr{P},\max}+O(1) \ \text{near} \ \mathbf{z}_i, \ \forall i=1,\ldots,p \big\}
    \end{split}
\end{flalign*}
for any $z\in D_j$. Then $\Phi^{D_j} _{\mathscr{P},\max}$ is a global Zhou weight related to $\mathscr{P}$ on $D_j$, and $\Phi_{\mathscr{P},\max}^{D_j}=\Phi_{\mathscr{P},\max}^D+O(1)$ near $\mathbf{z}_i$ for any $i$ and $j$ (see Proposition \ref{prop-local.produce.global}).

\begin{Lemma}\label{lem-approx.D-j.inside}
    The sequence $\left(\Phi_{\mathscr{P},\max}^{D_{-j}}\right)_{j\ge j_0}$ converges to $\Phi_{\mathscr{P},\max}^D$ pointwisely.
\end{Lemma}

\begin{proof}
    Observe that $\big(\Phi^{D_{-j}} _{\mathscr{P},\max}\big)_{j\ge j_0}$ is a decreasing sequence of negative plurisubharmonic functions, then the pointwise limit $\varPhi:=\lim_{j\to\infty} \Phi^{D_{-j}} _{\mathscr{P},\max}$ exists, and $\varPhi\in\mathrm{PSH}^-(D)$. In addition, we have $\varPhi\ge\Phi_{\mathscr{P},\max}^D$ since $\Phi_{\mathscr{P},\max}^{D_{-j}}\ge\Phi_{\mathscr{P},\max}^D$ on $D_{-j}$ for every $j\ge j_0$, and $(\mathbf{f}_{0,i}, \mathbf{z}_i)\notin \mathcal{I}(u_{0,i}+\varPhi)_{\mathbf{z}_i}$ for $i=1,\ldots,p$ since $\varPhi\le \Phi_{\mathscr{P},\max}^{D_{-j_0}}$. Then according to the definition of $\Phi_{\mathscr{P},\max}^D$, we get $\varPhi=\Phi_{\mathscr{P},\max}^D$ on $D$.
\end{proof}

\begin{Lemma}\label{lem-approx.Dj.outside}
    The sequence $\left(\Phi_{\mathscr{P},\max}^{D_{j}}\right)_{j\ge 1}$ converges to $\Phi_{\mathscr{P},\max}^D$ uniformly on $\overline{D}$, where $\Phi_{\mathscr{P},\max}^{D_j}\big|_{\partial D_j}$ and $\Phi_{\mathscr{P},\max}^D\big|_{\partial D}$ are defined to be $0$.
\end{Lemma}

\begin{proof}
    Observe that $\left(\Phi_{\mathscr{P},\max}^{D_{j}}\right)_{j\ge 1}$ is an increasing sequence of negative functions on $\overline{D}$. Let
    \[\mathsf C_j:=\inf_{z\in\partial D}\Phi_{\mathscr{P},\max}^{D_j}(z)\in (-\infty,0), \ \forall j\ge 1.\]

    First, we verify $\lim_{j\to\infty}\mathsf C_j=0$. Let $r>0$ sufficiently small such that
    \[\bigcup_{i=1}^p\mathbb{B}^n(\mathbf{z}_i,r)=\bigsqcup_{i=1}^p\mathbb{B}^n(\mathbf{z}_i,r)\subset\subset D,\]
    where `$\bigsqcup$' denotes the disjoint union. Since $\Phi_{\mathscr{P},\max}^{D_1}\in L_{\mathrm{loc}}^{\infty}\big(D_1\setminus\{\mathbf{z}_1,\ldots,\mathbf{z}_p\}\big)$ and $\varrho$ is continuous on $D_1=\Omega$, we can choose some $A>0$ such that
    \[\Phi_{\mathscr{P},\max}^{D_1}(z)\ge A \varrho(z), \ \forall z\in \bigsqcup_{i=1}^p\partial\mathbb{B}^n(\mathbf{z}_i,r).\]
    For any $j\ge 1$, set
    \begin{equation*}
		\phi_j(z):=\left\{
		\begin{array}{ll}
			  \Phi_{\mathscr{P},\max}^{D_j}(z)-A/j& \ \text{for} \ z\in\bigsqcup_{i=1}^p\mathbb{B}^n(\mathbf{z}_i,r), \\
             \max\left\{\Phi_{\mathscr{P},\max}^{D_j}(z), A\varrho(z)\right\}-A/j & \ \text{for} \ z\in D_j\setminus\bigsqcup_{i=1}^p\mathbb{B}^n(\mathbf{z}_i,r).
		\end{array}
		\right.
	\end{equation*}
    Then $\phi_j\in\mathrm{PSH}^-(D_j)$, and $\phi=\Phi_{\mathscr{P},\max}^D+O(1)$ near each $\mathbf{z}_i$. By the definition of $\Phi_{\mathscr{P},\max}^{D_j}$, we have $\Phi_{\mathscr{P},\max}^{D_j}\ge\phi_j$ on $D_j$. It follows that
    \[\Phi_{\mathscr{P},\max}^{D_j}(z)\ge A\varrho(z)-\frac{A}{j}=-\frac{A}{j}, \ \forall z\in \partial D.\]
    Thus, $\mathsf C_j\ge -A/j$, and we get $\lim_{j\to\infty}\mathsf C_j=0$.
    
    Next, we prove that $\left(\Phi_{\mathscr{P},\max}^{D_{j}}\right)_{j\ge 1}$ converges to $\Phi_{\mathscr{P},\max}^D$ uniformly on $\overline{D}$. For any $j\ge 1$, define
	\begin{equation*}
		\Phi_j(z):=\left\{
		\begin{array}{ll}
			\Phi_{\mathscr{P},\max}^{D_j}(z) & z\in D_j\setminus D, \\
			\max\left\{\Phi_{\mathscr{P},\max}^{D_j}(z), \Phi_{\mathscr{P},\max}^D(z)+\mathsf C_j\right\} & z\in\overline{D}.
		\end{array}
		\right.
	\end{equation*}
	Then $\Phi_j\in \mathrm{PSH}^-(D_j)$ is continuous near $\partial D$, and $\Phi_j\ge \Phi_{\mathscr{P},\max}^{D_j}$ on $D_j$. On the other hand, we have
	\[\Phi_j=\Phi_{\mathscr{P},\max}^{D_j}+O(1) \ \text{near} \ \mathbf{z}_i, \ \forall i=1,\ldots,p,\]
	which implies $\Phi_j\le \Phi_{\mathscr{P},\max}^{D_j}$. Consequently, $\Phi_{\mathscr{P},\max}^D+\mathsf C_j\le \Phi_{\mathscr{P},\max}^{D_j}$ on $\overline{D}$. Thus,
	\[\Phi_{\mathscr{P},\max}^{D_j}\le \Phi_{\mathscr{P},\max}^D\le \Phi_{\mathscr{P},\max}^{D_j}-\mathsf C_j \ \text{on} \ \overline{D}.\]
	Since $\lim_{j\to\infty} \mathsf C_j=0$, the lemma is proved.
\end{proof}

We may denote $D=D_{\infty}$ in the following for convenience. For any $j\in\mathbb{Z}_{>0}\cup\mathbb{Z}_{\le -j_0}\cup\{\infty\}$, and $m\in\mathbb{N}_+$, denote
\begin{flalign*}
    \begin{split}
        \mathscr{E}_{m}(D_j)&:=\big\{f\in\mathcal{O}(D_j) : \sup_{z\in D}|f(z)|\le 1, (f,\mathbf{z}_i)\in\mathcal{I}\big(m\Phi^{D_j}_{\mathscr{P},\max}\big)_{\mathbf{z}_i}, \ \forall i=1,\ldots,p\big\},\\
        \mathscr{A}^2_{m}(D_j)&:=\big\{f\in\mathcal{O}(D_j) : \|f\|_{D_j}\le 1, (f,\mathbf{z}_i)\in\mathcal{I}\big(m\Phi^{D_j}_{\mathscr{P},\max}\big)_{\mathbf{z}_i}, \ \forall i=1,\ldots,p\big\}. 
    \end{split}
\end{flalign*}
For any $z\in D_j$, denote
\begin{flalign*}
    \begin{split}
        \phi_{m,D_j}(z)&:=\sup_{f\in\mathscr{E}_m(D_j)}\frac{1}{m}\log|f(z)|,\\
\varphi_{m,D_j}(z)&:=\sup_{f\in\mathscr{A}^2_m(D_j)}\frac{1}{m}\log|f(z)|,
    \end{split}
\end{flalign*}
and
\[h_{m,D_j}(z):=\sup\left\{\frac{1}{m}\log|f(z)| : f\in\mathcal{O}(D_j), \ \int_{D_j}|f|^2e^{-2m\Phi_{\mathscr{P},\max}^{D_j}}\le 1\right\},\]
where $h_{m,D_j}$ is actually the function appearing in Demailly's approximation theorem (Lemma \ref{lem-Dem.approx}).

\begin{Lemma}\label{lem-approx.inequalities}
For any $m\in\mathbb{N}_+$ and $j\in\mathbb{Z}_{>0}\cup\mathbb{Z}_{\le -j_0}\cup\{\infty\}$, we have

	(1) $\phi_{m,D_j}$ and $\varphi_{m,D_j}$ are continuous and plurisubharmonic functions on $D_j$, and $\phi_{m,D_j}$ takes values in $[-\infty, 0)$.

    (2) Suppose that $k,l\in\mathbb{Z}_{>0}\cup\mathbb{Z}_{\le -j_0}\cup\{\infty\}$ satisfy $D_k\subsetneq D_l$, then
	\begin{equation}
        \varphi_{m,D_l}(z)\le \phi_{m,D_k}(z)+\frac{\log(c\delta(k,l)^{-n})}{m}, \ \forall z\in D_k,
    \end{equation}
	where $c$ is a constant only depending on $n$, $\delta(k,l)=\mathrm{dist}(\overline{D_k},\partial D_l)$.

    (3) For any $z\in D_j$,
    \begin{equation}
        h_{m,D_j}(z)\le\varphi_{m,D_j}(z).
    \end{equation}

    (4) There exists a positive constant $C_1$ independent of $m$ and $j$ such that
    \begin{equation}
        h_{m,D_j}(z)\ge \Phi_{\mathscr{P},\max}^{D_j}(z)-\frac{C_1}{m}, \ \forall z\in D_j.
    \end{equation}

    (5) There exists a constant $C_2$ independent of $m$ and $j$ such that
	\begin{equation}\label{eq-phi.m.Dj.le.m-C/m.PhiP}
		\phi_{m,D_j}(z)\le \frac{m-C_2}{m}\Phi^{D_j}_{\mathscr{P},\max}(z), \ \forall z\in D_j.
	\end{equation}
\end{Lemma}

In fact, the proof of Lemma \ref{lem-approx.inequalities} is as same as which in \cite{BGMY23}. We collate it here for the convenience of readers.
\begin{proof}
   (1) By the definitions of $\varphi_{m,D_j}$ and $\phi_{m,D_j}$, they are lower-semicontinuous on $D_j$, and $\phi_{m,D_j}$ is negative. Using Montel's theorem and the fact that $\mathscr{E}_{m}(D_j)$ and $\mathscr{A}^2_m(D_j)$ are compact subsets of $\mathcal{O}(D)$, we get that $\varphi_{m,D_j}$ and $\phi_{m,D_j}$ are upper-semicontinuous on $D_j$, thus they are continuous on $D_j$. In addition, it follows that they are plurisubharmonic on $D_j$.

   (2) Let $f\in \mathscr{A}^2_m(D_l)$, then there exists some $w\in \overline{D_k}$ such that $|f(w)|=\sup_{\overline{D}_k}|f|$. According to the mean value inequality applied to the plurisubharmonic function $|f|^2$, we have
   \begin{equation*}
       |f(w)|^2\le\frac{c^2}{\delta(k,l)^{2n}}\int_{\mathbb{B}^n(w,\delta(k,l))}|f(z)|^2\le\frac{c^2}{\delta(k,l)^{2n}}\int_{D_l}|f(z)|^2.
   \end{equation*}
   Thus,
   \[\varphi_{m,D_l}(z)\le \phi_{m,D_k}(z)+\frac{\log(c\delta(k,l)^{-n})}{m}, \ \forall z\in D_k.\]
   
   (3) Since $\Phi_{\mathscr{P},\max}^{D_j}$ is negative on $D_j$, for any $f\in A^2\big(D_j, 2m\Phi_{\mathscr{P},\max}^{D_j}\big)$, we have $f\in \mathcal{O}(D_j)\cap L^2(D_j)$, $(f,\mathbf{z}_i)\in \mathcal{I}\big(m\Phi^{D_j} _{\mathscr{P},\max}\big)_{\mathbf{z}_i}$ for $i=1,\ldots p$, and
   \[\int_{D_j}|f|^2\le\int_{D_j}|f|^2e^{-2m\Phi_{\mathscr{P},\max}^{D_j}}.\]
   Consequently, we obtain $h_{m,D_j}\le\varphi_{m,D_j}$ on $D_j$.

   (4) The inequality comes from Demailly's approximation theorem (Lemma \ref{lem-Dem.approx}). Since all $D_j$ are contained in a bounded domain $\Omega$, the constant $C_1$ can be selected to be independent of $m$ and $j$ (only dependent on the diameter of $\Omega$).

   (5) Let $f\in\mathscr{E}_m(D_j)$. Then $c^{f}_{\mathbf{z}_i}\big(\Phi_{\mathscr{P},\max}^{D_j}\big)\ge m$, $\forall i$, followed by
   \[\sigma_{\mathbf{z}_i}\left(\log|f|, \Phi_{\mathscr{P},\max}^{D_j}\right)\ge m-C_2, \ \forall i\in\{1,\ldots,p\}, \ m\in\mathbb{N}_+,\]
   according to Theorem \ref{thm-jump.no.Zhou.no.control}, where $C_2$ is a constant independent of $m$ and $j$. Thus, using Proposition \ref{prop-psi.le.sigmaZ.Phi}, we obtain
   \[\log|f|\le\frac{m-C_2}{m}\phi_{\mathscr{P},\max}^{D_j}\]
   on $D_j$ for each $f\in\mathscr{E}_m(D_j)$. Then it follows that inequality (\ref{eq-phi.m.Dj.le.m-C/m.PhiP}) holds.
\end{proof}

Now we give the proof of Theorem \ref{thm-approximation}.

\begin{proof}[\textbf{Proof of Theorem \ref{thm-approximation}}]
    First, we prove the estimates of the Zhou numbers of $\phi_m$ and $\varphi_m$ near each $\mathbf{z}_i$. According to Lemma \ref{lem-approx.inequalities}(2-5), we have
    \begin{equation}\label{eq-proof.approx.phim}
        \Phi_{\mathscr{P},\max}^{D_j}(z)-\frac{C_1}{m}-\frac{\log(c\delta(\infty,j)^{-n})}{m}\le\phi_m(z)\le\left(1-\frac{C_2}{m}\right)\Phi_{\mathscr{P},\max}^{D}(z),
    \end{equation}
    for any $z\in D$, $m\in\mathbb{N}_+$, and $j\ge 1$. Note that $\Phi_{\mathscr{P},\max}^{D_j}=\Phi_{\mathscr{P},\max}^D+O(1)$ near each $\mathbf{z}_i$. We get
    \[1-\frac{C_2}{m}\le\sigma_{\mathbf{z}_i}\big(\phi_m,\Phi_{\mathscr{P},\max}^D\big)\le 1, \ \forall i=1,\ldots,p.\]
    According to Lemma \ref{lem-approx.inequalities} (2-5), we also have
    \begin{equation}\label{eq-proof.approx.varphim}
        \Phi_{\mathscr{P},\max}^{D}(z)-\frac{C_1}{m}\le\varphi_m(z)\le\left(1-\frac{C_2}{m}\right)\Phi_{\mathscr{P},\max}^{D_{-k}}(z)+\frac{\log(c\delta(-k,\infty)^{-n})}{m},
    \end{equation}
    for any $z\in D$, $m\in\mathbb{N}_+$, and $k\ge j_0$. Thus, we also get
    \[1-\frac{C_2}{m}\le\sigma_{\mathbf{z}_i}\big(\varphi_m,\Phi_{\mathscr{P},\max}^D\big)\le 1, \ \forall i=1,\ldots,p.\]

    Next, we prove that $\phi_m$ and $\varphi_m$ converge to $\Phi_{\mathscr{P},\max}^D$ on $D$ pointwisely. Letting $m\to\infty, j\to\infty$ in equation (\ref{eq-proof.approx.phim}), by Lemma \ref{lem-approx.Dj.outside}, we get that $\phi_m$ converges to $\Phi_{\mathscr{P},\max}^D$ on $D$ pointwisely. Letting $m\to\infty, k\to\infty$ in equation (\ref{eq-proof.approx.varphim}), the same result holds for $\varphi_m$ by Lemma \ref{lem-approx.D-j.inside}.

\end{proof}

   \section{Convergence of global Zhou weights with distinct poles: proof of Theorem \ref{thm-Zvarepsilon.to.Z}}\label{sec-convergence}

    Following the method that Demailly used to prove the continuity of pluricomplex Green function on poles (Lemma 4.13 in \cite{Dem87b}), we give the following lemma, where Theorem \ref{thm-Zvarepsilon.to.Z} is a direct result of this lemma.  

    \begin{Lemma}
        Let $\mathbf{W}:=\{w_1,\ldots,w_p\}$ be a fixed set of $p$ distinct points in the bounded hyperconvex domain $D$. Then for any $\delta>0$ and any open subset $U$ of $D$ satisfying $U\supset \mathbf{W}$, there exist $U_1,\ldots,U_p\subset\subset U$ with $w_i\in U_i$, such that for any two sets of $p$ distinct points $\mathbf{Z}=\{\mathbf{z}_i\}_{i=1}^p$, $\mathbf{Z'}=\{\mathbf{z}'_i\}_{i=1}^p$ with $\mathbf{z}_i,\mathbf{z}'_i\in U_i$ for any $i$, and any $\zeta\in \overline{D}\setminus U$, we have
        \[(1+\delta)^{-1}\Phi_{\mathbf{Z},\max}^D(\zeta)\le\Phi_{\mathbf{Z'},\max}^D(\zeta)\le (1+\delta)\Phi_{\mathbf{Z},\max}^D(\zeta).\]
    \end{Lemma}

\begin{proof}
    Without loss of generality, we may choose $r>0$ sufficiently small such that
    \[\bigcup_{i=1}^p\mathbb{B}^n(w_i,2r)=\bigsqcup_{i=1}^p\mathbb{B}^n(w_i,2r)\subset\subset D,\]
    and assume that
    \[U=\bigsqcup_{i=1}^p\mathbb{B}^n(w_i,r).\]

    Let $\Psi_i\in\mathrm{PSH}^-\big(\mathbb{B}^n(o,2r)\big)$ be the unique global Zhou weight on $\mathbb{B}^n(o,2r)$ related to $|\mathbf{f}_{0,i}|^2e^{-2u_{0,i}}$ with $\Psi_i=\phi_{i,o,\max}+O(1)$ near $o$ for any $i=1,\ldots,p$. Then it is clear that for any $i\in\{1,\ldots,p\}$, and $\mathbf{z}_i\in\mathbb{B}^n(w_i,r)$,
    \[\Phi^D_{\mathbf{Z},\max}(\zeta)\le \Psi_i(\zeta-\mathbf{z}_i), \ \text{for  }  \ \zeta\in\mathbb{B}^n(w_i,r).\]
    
    On the other hand, let $R>0$ sufficiently large such that $D\subset\subset\mathbb{B}^n(o,R/2)$, and let $\Phi_i\in\mathrm{PSH}^-\big(\mathbb{B}^n(o,R)\big)$ be the unique global Zhou weight on $\mathbb{B}^n(o,R)$ related to $|\mathbf{f}_{0,i}|^2e^{-2u_{0,i}}$ with $\Phi_i=\phi_{i,o,\max}+O(1)$ near $o$ for any $i=1,\ldots,p$. Then it is also clear that
    \[\Phi^D_{\mathbf{Z},\max}(\zeta)\ge \sum_{i=1}^p\Phi_i(\zeta-\mathbf{z}_i), \ \text{for \ } \mathbf{z}_i\in\mathbb{B}^n(w_i,r), \ \zeta\in\overline{D}.\]

    Let $T>0$ sufficiently large such that
    \[\big\{z\in\mathbb{B}^n(o,R) : \Phi_i(z)\le -T\big\}\subset\subset\mathbb{B}^n(o,r), \ \forall i=1,\ldots,p,\]
    and
    \[(1+\delta)\sup_{\{z: \Phi_i(z)=-T\}}\Psi_i(z)<-T+\sum_{j\ne i}\inf_{r\le|z|<R}\Phi_j(z)\]
    for any $i\in\{1,\ldots,p\}$. Such $T$ exists since $\Phi_i\ge N_i\log|z|+O(1)$ near $o$ for some $N_i>0$, and $\Psi_i=\Phi_i+O(1)$ near $o$ for any $i$. Let
    \[U_i=\mathbb{B}^n(w_i,\eta/4),\ i=1,\ldots,p,\]
    where $\eta>0$ is sufficiently small such that
    \[(1+\delta)\sup_{z\in\{\Phi_i= -T\}+\mathbb{B}^n(o,\eta)}\Psi_i(z)<\inf_{z\in\{\Phi_i= -T\}+\mathbb{B}^n(o,\eta)}\Phi_i(z)+\sum_{j\ne i}\inf_{r\le|z|<R}\Phi_j(z)\]
    for any $i=1,\ldots,p$. Here we denote by
    \[\{\Phi_i= -T\}+\mathbb{B}^n(o,\eta):=\big\{z+a : \Phi_i(z)=-T, \ a\in\mathbb{B}^n(o,\eta)\big\}.\]
    Then for any $\mathbf{Z}=\{\mathbf{z}_i\}_{i=1}^p$, $\mathbf{Z'}=\{\mathbf{z}_i'\}_{i=1}^p$ with $\mathbf{z}_i,\mathbf{z}'_i\in U_i$ for any $i$, it holds that
    \[(1+\delta)\Phi^D_{\mathbf{Z},\max}(\zeta)<\sum_{i=1}^p\Phi_i(\zeta-\mathbf{z}'_i) \ \text{on} \ \bigcup_{i=1}^p\big\{\zeta\in D : \Phi_i(\zeta-w_i)=-T\big\}.\]
    Define
    \begin{equation*}
        v(\zeta):=\left\{
    \begin{array}{ll}
        \sum\limits_{i=1}^p\Phi_i(\zeta-\mathbf{z}'_i) & \text{if} \ \exists i,\ \Phi_i(\zeta-w_i)\le -T \\
        \max\left((1+\delta)\Phi^D_{\mathbf{Z},\max}(\zeta), \sum\limits_{i=1}^p\Phi_i(\zeta-\mathbf{z}'_i)\right) & \text{if} \ \forall i,\ \Phi_i(\zeta-w_i)\ge -T
    \end{array}
        \right.
    \end{equation*}
    for any $\zeta\in D$, then $v\in\mathrm{PSH}^-(D)$. Note that $v=\phi_{o,i,\max}+O(1)$ near $\mathbf{z}'_i$ for any $i$. By the definition of $\Phi_{\mathbf{Z}',\max}^D$, we get
    \[\Phi_{\mathbf{Z}',\max}^D(\zeta)\ge v(\zeta)\ge (1+\delta)\Phi_{\mathbf{Z},\max}^D(\zeta)\]
    for any $\mathbf{Z}=\{\mathbf{z}_i\}_{i=1}^p$, $\mathbf{Z}'=\{\mathbf{z}'_i\}_{i=1}^p$ with $\mathbf{z}_i,\mathbf{z}'_i\in U_i$ for any $i$, and any $\zeta\in \overline{D}\setminus U$. By changing the roles of $\mathbf{Z}$ and $\mathbf{Z}'$, we can also get
    \[\Phi_{\mathbf{Z},\max}^D(\zeta)\ge (1+\delta)\Phi_{\mathbf{Z}',\max}^D(\zeta)\]
    for $\mathbf{Z}$, $\mathbf{Z}'$ and $\zeta$ satisfying the same conditions. Consequently, the lemma holds.
\end{proof}

   \section{A semi-continuity result of Zhou numbers: proof of Theorem \ref{thm-semi.continuity}}\label{sec-semicontinuity}

   We need some lemmas. The following lemma is an analyticity theorem concerning relative types to maximal weights due to Rashkovskii. However, the settings of the following lemma is slightly different from the original result in \cite{Rash06}, so we will recall the proof in the Appendix for the convenience of readers.

   \begin{Lemma}[Theorem 7.1 in \cite{Rash06}, see also \cite{Rash09}]\label{Lem-Rash.ana.thm}
    Let $\Omega$ be a pseudoconvex domain in $\mathbb{C}^n$. Let $\varphi : \Omega\times\Omega\to [-\infty, +\infty)$ be a continuous plurisubharmonic function. Denote $\varphi_{\zeta}(x):=\varphi(x,\zeta)$. Assume that

    (i) $\{x : \varphi(x,\zeta)=-\infty\}=\{\zeta\}$;

    (ii) for every compact subset $K\subset\Omega$, there exists $R(K)\in (-\infty,+\infty)$ such that
    \[\big\{(x,\zeta)\in \Omega\times K : \varphi(x,\zeta)\le R(K)\big\}\subset\subset\Omega\times\Omega.\]

    (iii) $\big(dd^c\varphi\big)^n=0$ on $\{\varphi(x,\zeta)>-\infty\}$;

    (iv) $e^{\varphi(x,\zeta)}$ is  locall H\"{o}lder continuous in $\zeta$ on $\Omega\times\Omega$, i.e., for any $(x_0,\zeta_0)\in\Omega\times\Omega$, there exists a neighborhood $V\subset\Omega\times\Omega$ of $(x_0,\zeta_0)$ such that
    \[\left|e^{\varphi(x,\zeta_1)}-e^{\varphi(x,\zeta_2)}\right|\le M|\zeta_1-\zeta_2|^{\beta}, \ (x,\zeta_i)\in V, i=1,2,\]
    for some $M,\beta>0$.

    Then for any $u\in\mathrm{PSH}(\Omega)$, the set
    \[S_c(u,\varphi,\Omega)=\big\{\zeta\in\Omega : \sigma_{\zeta}(u,\varphi_{\zeta})\ge c\big\},\]
    is an analytic subset of $\Omega$ for each $c>0$. 
   \end{Lemma}

   \begin{Lemma}[Proposition 3.3 in \cite{Rash13}]\label{lem-iso.ana.sing.maximal}
    Any isolated analytic singularity is maximal.
   \end{Lemma}

   Lemma \ref{lem-iso.ana.sing.maximal} means that, for any plurisubharmonic function $u$ with analytic singularity near $o$ and $u^{-1}(-\infty)=\{o\}$, there exists a plurisubharmonic function $\tilde{u}$ on a neighborhood $V$ of $o$ such that $\tilde{u}=u+O(1)$ on $V$, and $(dd^c\tilde{u})^n=0$ on $V\setminus\{o\}$. 

   Here we need to repeat the proof of Lemma \ref{lem-iso.ana.sing.maximal} in \cite{Rash13} to help us to prove Theorem \ref{thm-semi.continuity}.

   \begin{proof}[\cite{Rash13} Proof of Lemma \ref{lem-iso.ana.sing.maximal}]
    Let $\varphi=c\log|\mathbf{F}|$, where $c>0$, and $\mathbf{F}$ is a holomorphic mapping from a neighborhood of $o$ to some $\mathbb{C}^N$. Then one can always find $n$ holomorphic functions $\xi_1,\ldots,\xi_n$, which are generic linear combinations of the components $F_j$ of $\mathbf{F}$ (see \cite{NR54}, or \cite{CADG} Chapter VIII Lemma 10.3), such that $\log|\mathbf{F}|=\log|\xi|+O(1)$, where $\xi=(\xi_1,\ldots,\xi_n)$. By King's formula, $(dd^c\log|\xi|)^n=0$ outside the zero set of $\mathbf{F}$. In addition, $\varphi=c\log|\xi|+O(1)$ near $o$.
   \end{proof}
   Note that $e^{c\log|\xi|}$ is H\"{o}lder continuous near $o$.

Now we give the proof of Theorem \ref{thm-semi.continuity}.

\begin{proof}[\textbf{Proof of Theorem \ref{thm-semi.continuity}}]
    Let $a>0$ such that $\phi_{o,\max}\in\mathrm{PSH}^-(\mathbb{B}^n(o,4a))$. Denote $\{h_m\}$ be the sequence of Demailly's approximation (Lemma \ref{lem-Dem.approx}) of $\phi_{o,\max}$ on $\mathbb{B}^n(o,4a)$:
    \[h_m(z):=\sup\left\{\frac{1}{m}\log|f(z)| : f\in\mathcal{O}\big(\mathbb{B}^n(o,4a)\big), \ \int_{\mathbb{B}^n(o,4a)}|f|^2e^{-2m\phi_{o,\max}}\le 1\right\}.\]
    Using Lemma \ref{lem-approx.inequalities} (or inequality (\ref{eq-Phim.PhiP}) in the proof of Proposition \ref{prop-continuous.exhaustive}), for the local Zhou weight $\phi_{o,\max}$, we have that there exists positive constants $C_1$, $C_2$ and $C_3$ independent of $m$ such that
    \begin{equation}\label{eq-phiomax.hm.phiomax}
        \phi_{o,\max}(z)-\frac{C_1}{m}\le h_m(z)\le \left(1-\frac{C_2}{m}\right)\phi_{o,\max}(z)+C_3, \ \forall z\in \mathbb{B}^n(o,2a),
    \end{equation}
    for any $m\in\mathbb{N}_+$.
    
    The function $h_m$ has isolated analytic singularity near $o$. Lemma \ref{lem-iso.ana.sing.maximal} shows that we can find a maximal plurisubharmonic function $\tilde{h}_m$ near $o$, such that $\tilde{h}_m=h_m+O(1)$ near $o$. Thus, the relative type to $h_m$
    \[\sigma_o(\psi,h_m):=\sup\big\{b : \psi\le bh_m+O(1) \ \text{near} \ o\big\}\]
    coincides with the relative type to $\tilde{h}_m$. By the proof of Lemma \ref{lem-iso.ana.sing.maximal}, the function $\tilde{h}_m$ can be chosen to satisfy that $e^{\tilde{h}_m}$ is H\"{o}lder continuous near $o$. Since we only consider the relative types to $\tilde{h}_m$, without loss of generality, we assume that $\tilde{h}_m$ is negative, H\"{o}lder continuous, and maximal on $\mathbb{B}^n(o,2a)$. Now we deduce that the set $E_c(\psi,h_m)=E_c(\psi,\tilde{h}_m)$ is an analytic subset for any $\psi\in\mathrm{PSH}(D)$, $c>0$, and $m\in\mathbb{N}_+$ by Lemma \ref{Lem-Rash.ana.thm}.

    Denote
    \[\varphi(x,\zeta):=\tilde{h}_m(x-\zeta), \ \forall x,\zeta\in\mathbb{B}^n(o,a).\]
    Then it can be checked that $\varphi$ satisfies the statements (i)-(iv) in Lemma \ref{Lem-Rash.ana.thm} for $\Omega=\mathbb{B}^n(o,a)$. Thus, for any $u\in\mathrm{PSH}\big(\mathbb{B}^n(o,a)\big)$, the set
    \[S_{c}\big(u,\varphi,\mathbb{B}^n(o,a)\big)=\big\{\zeta\in\mathbb{B}^n(o,a) : \sigma_{\zeta}(u,\varphi_{\zeta})\ge c\big\}\]
    is an analytic subset of $\mathbb{B}^n(o,a)$ for each $c>0$. We conclude that
    \[E_c(u,\tilde{h}_m)=\big\{z\in \mathbb{B}^n(o,a) : \sigma_z(u,\tau_z\tilde{h}_m)\ge c\big\}\]
    is an analytic subset of $\mathbb{B}^n(o,a)$ for any $u\in\mathrm{PSH}\big(\mathbb{B}^n(o,a)\big)$ and $c>0$. It is also true for $h_m$. Since our desired result is local, for any domain $D$ in $\mathbb{C}^n$, we get that the set
    \[E_c(\psi,h_m):=\big\{z\in D : \sigma_z(\psi,\tau_zh_m)\ge c\big\}\]
    is an analytic subset of $D$ for any $\psi\in\mathrm{PSH}(D)$, $c>0$, and $m\in\mathbb{N}_+$.

    Finally, we prove that $E_c(\psi,\phi_{o,\max})$ is an analytic subset. Let $u$ be a plurisubharmonic function near $o$. According to equation (\ref{eq-phiomax.hm.phiomax}), we have that
    \[\left(1-\frac{C_2}{m}\right)\sigma_o(u,h_m)\le\sigma_o(u,\phi_{o,\max})\le\sigma_{o}(u,h_m)\]
    for any $m\in\mathbb{N}_+$. Thus, on any domain $D$, we can write
    \[E_c(\psi,\phi_{o,\max})=\bigcap_{m\ge 1}E_c(\psi,h_m).\]
    Now the fact that $E_c(\psi,h_m)$ are all analytic subsets of $D$ implies that $E_c(\psi,\phi_{o,\max})$ is an analytic subset of $D$.
\end{proof}

\begin{Remark}
    Since the only consequece of $\phi_{o,\max}$ being a local Zhou weight we used in the above proof is the holding of equation (\ref{eq-phiomax.hm.phiomax}), which also holds when $\phi_{o,\max}$ is just a tame maximal weight. Thus, replacing $\phi_{o,\max}$ by a tame maximal weight (see \cite{BFJ08}) near $o$, Theorem \ref{thm-semi.continuity} still holds.
\end{Remark}

\vspace{.1in} {\em Acknowledgements.}
The second named author was supported by National Key R\&D Program of China 2021YFA1003100, NSFC-11825101, NSFC-11522101 and NSFC-11431013. The forth named author was supported by China Postdoctoral Science Foundation BX20230402.

\section{Appendix: Proof of Lemma \ref{Lem-Rash.ana.thm}}\label{sec-appendix}
We give the proof of Lemma \ref{Lem-Rash.ana.thm} by doing some slight changes of the proof of Theorem 7.1 in \cite{Rash06}.
\begin{proof}[\textbf{Proof of Lemma \ref{Lem-Rash.ana.thm}}]
Since the result is local, we may assume $\Omega=\mathbb{B}^n(o,R)$ without loss of generality. Denote $\Omega':=\mathbb{B}(o,R/2)$. After addition of a constant to $\varphi$, with (ii), we may also assume that there exists a compact subset $K\subset \Omega$ such that
\[\{(x,\zeta)\in\Omega\times\Omega' : \varphi(x,\zeta)\le 0\}\subset K\times\Omega'.\]
By Proposition \ref{prop-MVyr.psh} \cite[Theorem 6.11]{Dem85} (set $X=\{\varphi<0\}$, $Y=\pi(X)$, $R\equiv 0$, and $A\equiv -\infty$ in Proposition \ref{prop-MVyr.psh}, where $\pi : (x,\zeta)\mapsto \zeta$), the function
\[\mathcal{U}(\zeta,t):=\Lambda(u,\varphi_{\zeta},\mathrm{Re\ }t)\]
is plurisubharmonic on $\big\{(\zeta,t)\in\Omega'\times\mathbb{C} : \mathrm{Re\ }t<0\big\}$, where the semi-continuity of $\mathcal{U}$ comes from the semi-continuity of $u$.

Given $a>0$, the function $\mathcal{U}(\zeta,t)-a\mathrm{Re\ }t$ is plurisubharmonic and independent of $\mathrm{Im\ }t$. By Kiselman's minimum principle (Proposition \ref{prop-Kmp}), the function
\[\mathcal{U}_a(\zeta):=\inf\big\{\Lambda(u,\varphi_{\zeta},r)-ar : r<0\big\}\]
is plurisubharmonic on $\Omega$.

Let $\zeta\in\Omega'$.  Note that
\[\sigma_{\zeta}(u,\varphi_{\zeta})=\lim_{r\to-\infty}r^{-1}\Lambda(u,\varphi_{\zeta},r).\]

If $a>\sigma_{\zeta}(u,\varphi_{\zeta})$, then there exists $r_a<0$ such that $\Lambda(u,\varphi_{\zeta},r)>ar$ for all $r\le r_a$. For $r_a<r<0$, we have
\[\Lambda(u,\varphi_{\zeta},r)-ar\ge \Lambda(u,\varphi_{\zeta},r_a).\]
Therefore, $\mathcal{U}_a(\zeta)>-\infty$.

Now suppose $a<\sigma_{\zeta}(u,\varphi_{\zeta})$. By (ii), one can take a compact subset $L\subset\Omega'$ such that $\zeta\in L$, and
\[\big\{(x,y)\in \Omega\times L : \varphi(x,y)\le R(L)\big\}\subset K\times L,\]
for some $R(L)<0$, and $e^{\varphi}$ is H\"{o}lder continuous in $\zeta$ on $K\times L$ with some constants $M,\beta$ as like in (iv), which implies that there exists a neighborhood $V_{\zeta}\subset L$ of $\zeta$ and a constant $C_1>0$ such that
\begin{flalign*}
    \begin{split}
        \Lambda(u,\varphi_{\zeta'},r)\le&\Lambda\big(u,\varphi_{\zeta},\log(e^{r}+M|\zeta'-\zeta|^{\beta})\big)\\
        \le&\sigma_{\zeta}(u,\varphi_{\zeta})\log\big(e^{r}+M|\zeta'-\zeta|^{\beta}\big)+C_1,
    \end{split}
\end{flalign*}
for any $\zeta'\in V_{\zeta}$ and $r<R(L)-1$. Denote $r_{\zeta'}:=\beta\log|\zeta'-\zeta|$. If $\zeta'\in V_{\zeta}$ and $r_{\zeta'}<R(L)-1$, then
\begin{flalign}\label{eq-proof.Rash.ana}
    \begin{split}
        \mathcal{U}_a(\zeta')\le&\Lambda(u,\varphi_{\zeta'},r_{\zeta'})-ar_{\zeta'}\\
        \le&\big(\sigma_{\zeta}(u,\varphi_{\zeta})-a\big)\beta\log|\zeta'-\zeta|+C_2,
    \end{split}
\end{flalign}
for some constant $C_2$.

Given $a,b>0$, let
\[\mathcal{Z}_{a,b}:=\left\{\zeta\in\Omega' : e^{-\mathcal{U}_a/b} \ \text{is not integrable near} \ \zeta\right\}.\]
It follows from a consequece of H\"{o}rmander-Bombieri-Skoda theorem (Proposition \ref{prop-HBS}) that the sets $\mathcal{Z}_{a,b}$ are all analytic subsets of $\Omega'$. In addition, we have $\mathcal{Z}_{a',b'}\supset \mathcal{Z}_{a'',b''}$ if $a'\le a''$, $b'\le b''$.

If $\zeta\notin S_c(u,\varphi,\Omega')$ and $\sigma_{\zeta}(u,\varphi_{\zeta})<a<c$, then $\mathcal{U}_a(\zeta)>-\infty$, so $\zeta\notin\mathcal{Z}_{a,b}$ for all $b>0$ by Skoda's theorem (Proposition \ref{prop.Skoda.thm}).

If $\zeta\in S_c(u,\varphi,\Omega')$, $a<c$ and $b<(c-a)\beta(2n)^{-1}$, then (\ref{eq-proof.Rash.ana}) implies $\zeta\in\mathcal{Z}_{a,b}$.

Thus, we can write
\[S_c(u,\varphi,\Omega')=\bigcap_{a<c, \ b<\frac{(c-a)\beta}{2n}}\mathcal{Z}_{a,b}.\]
Therefore, $S_c(u,\varphi,\Omega')$ is an analytic subset of $\Omega$, and so is $S_c(u,\varphi,\Omega)$ since the result is local.
\end{proof}

We list below some propositions we used in the above proof.

\begin{Proposition}[see Theorem 6.11 in \cite{Dem85}]\label{prop-MVyr.psh}
    Let $\pi : X\to Y$ be a morphism between analytic spaces with pure dimensions $\dim X=m+n$, $\dim Y=m$, $\varphi : [-\infty, +\infty)$ a continuous plurisubharmonic function, and $R : Y\to (-\infty,+\infty]$ (resp. $A : Y\to [-\infty,+\infty)$) a lower-semicontinuous (resp. upper-semicontinuous) function satisfying:

    (a) $\pi$ is surjective, and for any $y\in Y$, the fiber $\pi^{-1}(y)$ has pure dimension $n$;

    (b) $\pi$ is a Stein morphism, i.e. $Y$ has an open covering $(\Omega_j)_{j\in J}$, such that $\pi^{-1}(\Omega_j)$ is Stein for any $j\in J$.

    (c) $\varphi(x)<R(\pi(x))$ and $A(y)<R(y)$ for any $x\in X$, $y\in Y$.

    (d) For any $y\in Y$ and $r<R(y)$, there exists a neighborhood $U$ of $y$ in $Y$ such that $\pi^{-1}(U)\cap B(r)\subset\subset X$.

    (e) $(dd^c\varphi)^n\equiv 0$ on the open set $\big\{x\in X : \varphi(x)>A(\pi(x))\big\}$.

    Denote $B(r)=\{\varphi<r\}$. For any plurisubharmonic function $V$ on $X$, $y\in Y$, and $r\in (-\infty, R(y))$, let
    \[M^{\infty}_{V}(y,r)=\sup_{\pi^{-1}(y)\cap B(r)}V.\]

    Then for fixed $r$, the funtion $y\mapsto M^{\infty}_V(y,r)$ is a weakly plurisubharmonic function on the open set $\big\{y\in Y : A(y)<r<R(y)\big\}$.
\end{Proposition}

A locally integrable function $u$ on a complex space $X$ with pure dimension $n$ is called a weakly plurisubharmonic function, if $dd^c u\ge 0$ in the sense of currents. For upper-semicontinuous functions on a complex manifold $X$, the weakly plurisubharmonic functions on $X$ coincides with the plurisubharmonic functions.

\begin{Proposition}[Kiselman's minimum principle, \cite{Ki78}]\label{prop-Kmp}
Let $\Omega\subset\mathbb{C}^m\times\mathbb{C}^n$ be a pseudoconvex open set such that each slice
\[\Omega_{\zeta}=\big\{z\in\mathbb{C}^n : (\zeta,z)\in\Omega\big\}, \ \zeta\in\mathbb{C}^m,\]
is a convex tube $\omega_{\zeta}+i\mathbb{R}^n$, $\omega_{\zeta}\subset\mathbb{R}^n$. For every plurisubharmonic function $v(\zeta,z)$ on $\Omega$ that does not depend on $\mathrm{Im\ }z$, the function
\[u(\zeta)=\inf_{z\in\Omega_{\zeta}}v(\zeta,z)\]
is plurisubharmonic of locally $\equiv-\infty$ on $\Omega'=\mathrm{pr}_{\mathbb{C}^n}(\Omega)$.
\end{Proposition}

\begin{Proposition}[Consequece of H\"{o}rmander-Bombieri-Skoda theorem, see Corollary 7.7 in \cite{CADG}]\label{prop-HBS}
Let $u$ be a plurisubharmonic function on a complex manifold $X$. The set of points in a neighborhood of which $e^{-u}$ is not integrable is an analytic subset of $X$.
\end{Proposition}

\begin{Proposition}[Skoda's theorem, \cite{Sko72}]\label{prop.Skoda.thm}
Let $u$ be a plurisubharmonic function near $z\in\mathbb{C}^n$. If the Lelong number $\nu(u,z)<1$, then $e^{-2u}$ is integrable near $z$.
\end{Proposition}

\bibliographystyle{alpha}
\bibliography{xbib}

\end{document}